\documentclass[12pt,reqno]{amsart}

\usepackage[shortlabels]{enumitem}
\usepackage{esint}
\usepackage{physics}
\usepackage{relsize}
\usepackage[backref]{hyperref}
\usepackage{mathtools}
\usepackage{amsfonts}
\usepackage[all]{xy}
\usepackage{geometry}
\usepackage{appendix}
\usepackage{stackengine}
\usepackage{esint}
\usepackage{mathtools}
\usepackage{amssymb}
\usepackage{xcolor}
\usepackage{tikz-cd}
\newtheorem{theorem}{Theorem}[section]
\newtheorem{lemma}[theorem]{Lemma}
\newtheorem{definition}[theorem]{Definition}

\theoremstyle{corollary}
\newtheorem{corollary}[theorem]{Corollary}
\newtheorem{conjecture}[theorem]{Conjecture}

\theoremstyle{assumption}

\theoremstyle{proposition}
\newtheorem{proposition}[theorem]{Proposition}
\theoremstyle{remark}
\newtheorem{remark}[theorem]{Remark}
\numberwithin{equation}{section}

\newcommand{\pp}{\bar\partial}

\newcommand{\ch}{\mathrm{ch}}
\newcommand{\rk}{\mathrm{rank}}

\newcommand{\length}{\mathrm{length}}
\newcommand{\gr}{\mathrm{Gr}}
\everymath{\displaystyle}
 \date{\today}

\begin{document}

\title{Algebraic limits of Hermitian-Yang-Mills connections on K\"ahler surfaces}
\author{Keshu Zhou}
\address{Department of Mathematics, University of California, Berkeley, CA 94720, USA} 
\email{keshu\_zhou@berkeley.edu}
\thanks{}

\begin{abstract}
We develop a geometric way to uniformly carry out H\"ormander's $L^2$-estimates for Hermitian-Yang-Mills connections with bounded energy on K\"ahler surfaces. We then show the Uhlenbeck convergence of \textit{smooth} HYM connections on K\"ahler surfaces can be realized as algebraic convergence inside an analytic flat family. As applications, we prove a continuity theorem for smoothable families and a boundedness result for $\mu$-semistable sheaves with fixed topology. We also discuss potential applications for the moduli problem.
\end{abstract}

\maketitle
\tableofcontents

\section{Introduction}

Let $(X,\omega)$ be a K\"ahler surface that is not necessarily compact, and $(E,h)\to X$ be a complex Hermitian vector bundle. Let $\{A_j\}$ be a sequence of \textit{smooth} Hermitian-Yang-Mills connections (abbreviated HYM below) on $(E,h)$ with bounded energy and Einstein constant, i.e. $A_j$ are unitary connections satisfying
\[
\sqrt{-1}\Lambda_\omega F_{A_j}=\mu_j\mathrm{Id}_{E},\quad F_{A_j}^{0,2}=0,
\]
and
\[
\sup_j\int_X\vert F_{A_j}\vert^2d\mu_g\leq\Lambda,\quad\sup_j\vert\mu_j\vert\leq\mu.
\]

By Uhlenbeck compactness, after possibly passing to a subsequence we may assume it converges weakly to an ideal HYM connection $(A_\infty,\mathcal{C}_\infty)$. The collection of connections $\{A_j\}_{j=1}^{\infty}$ naturally defines a collection of holomorphic vector bundles $\mathcal E_j$, $\mathcal E_\infty$ on $X$, and we are concerned with realizing Uhlenbeck convergence as algebraic convergence inside a flat analytic family. 

Our first main result achieves this \textit{locally}:
\begin{theorem}\label{thm--local convergence in a flat family}
    By passing to a subsequence and shrinking the domain $X$ to a coordinate ball $U$ with a unique bubbling point $0$, the following holds. There exists a coherent subsheaf $\mathcal F_\infty\subset\mathcal E_\infty$ such that
    \[
    \mathcal F_\infty\vert_{U\setminus\{0\}}=\mathcal E_\infty\vert_{U\setminus\{0\}}.
    \]
    Moreover, given any local presentation
    \[
    0\to\mathcal O_U^k\xrightarrow{R_\infty}\mathcal O_U^N\xrightarrow{q_\infty}\mathcal F_\infty\to 0,
    \]
    after shrinking to a nested sequence of coordinate balls
    \[
    B^{\prime}\subset\subset B\subset\subset B^+\subset\subset U,
    \]
 there exist holomorphic morphisms $q_j:\mathcal O_{B^+}^N\to\mathcal E_j$ and holomorphic matrices $R_j$ on $B$, such that
 \begin{itemize}
     \item Columns of $q_j$ converge to those of $q_\infty$ in the Uhlenbeck sense;
     \item For $j$ sufficiently large, we have exact sequences of sheaves over $B^{\prime}$,
     \[
     0\to\mathcal O_{B^{\prime}}^k\xrightarrow{R_j}\mathcal O_{B^{\prime}}^N\xrightarrow{q_j}\mathcal E_j\vert_{B^{\prime}}\to 0.
     \]
     Moreover, we can assume $R_j\to R_\infty$ in $C^{\infty}(\overline{B^{\prime}})$;
     \item There exists a uniform constant $C$ independent of $j$, such that for all $j$ large and every holomorphic $L^2$-section of $\mathcal E_j$ over $B$, we can find $a\in\Gamma(B,\mathcal O_B^N)$, such that
     \[
     q_ja=s,\qquad \Vert a\Vert_{L^2(B)}\leq C\Vert s\Vert_{L^2(B)}.
     \]
     Here the $L^2$-norm on $\mathcal O^N_B$ is induced by standard Hermitian metric, and the $L^2$-norm on $\mathcal E_j$ is induced by given Hermitian metric $h$ on $E$.
 \end{itemize}
 As a result, the algebraic multiplicity is equal to the analytic multiplicity, i.e.
 \[
 \mathrm{length}_{\mathcal{O}_{U,0}}(\mathcal F_\infty^{\vee\vee}/\mathcal F_\infty)=\lim_{j\to\infty}\frac{1}{8\pi^2}\int_{U}\left(\vert F_{A_j}\vert^2-\vert F_{A_\infty}\vert^2\right)d\mu_g.
 \]
\end{theorem}

In particular, the algebraic limit $\mathcal F_\infty$ provides refined information for the Uhlenbeck convergence. It endows the bubbling locus with a scheme theoretical structure, while the gauge theoretical limit only memorizes its reduced structure as well as multiplicities. This also answers a question in \cite[Remark 2.13]{DS2} in the $2$-dimensional sheaf theoretical setting; see also \cite[Remark 2.1]{sun25}.

Globalizing Theorem \ref{thm--local convergence in a flat family}, we obtain:
\begin{theorem}\label{thm-global convergence in a flat family}
    Assume $(X,\omega)$ is a compact K\"ahler surface. By possibly passing to a subsequence, there exists a coherent torsion-free sheaf ${\mathcal{F}}_\infty$ satisfying the following:
    \begin{itemize}
        \item [(1)] The normalization\footnote{The standard terminology in sheaf theory would be reflexive hull, or double dual. We nevertheless adopt this term to connect with the theory of algebraic varieties; see the discussions in \cite[Section 8.1]{ZZ26}.} of $\mathcal F_\infty$ is $\mathcal E_\infty$, i.e. $\mathcal F_\infty^{\vee\vee}=\mathcal E_\infty$. In particular, $\mathcal F_\infty$ is $\mu$-semistable.
        \item [(2)] The algebraic cycle $\mathcal C(\mathcal F_\infty)$ is identical to $\mathcal C_\infty$.
    \end{itemize}
    Moreover, there exist a finite dimensional complex analytic space germ $(T,t_0)$ and an analytic flat family $\mathcal F\to X\times T$, such that for some sequence of points $t_j\to t_0$ in $T$, we have $\mathcal F_{t_j}\cong\mathcal E_j$ and $\mathcal F_{t_0}\cong\mathcal F_\infty$.
\end{theorem}

Such a flat convergence theorem was essentially obtained for polarized projective manifolds by the work of \cite{li1993,GSTW2021}. Our main contribution here is to drop any algebraicity condition on both $X$ and $[\omega]$. As an immediate corollary, we know that the Donaldson-Uhlenbeck boundary of compactified HYM moduli can be \textit{covered} by images of analytic spaces parameterizing semistable sheaves. This overcomes one of the main obstacles in the moduli problem \cite[Page 996, Item (3)]{BTT2017}.

\

As an application of the above results, we derive two essentially well-known geometric consequences towards the construction of a compact analytic moduli space. From now on, we assume $(X,\omega)$ is a compact K\"ahler surface. 

Consider the moduli space $M_{HYM}^*$ of smooth irreducible HYM connections on $E$ with a fixed determinant, which is known to be a complex analytic space by the celebrated Donaldson-Uhlenbeck-Yau theorem \cite{donaldson85,donaldson87,uhlenbeck-yau86,fujiki-schumacher87,miyajima89}. According to \cite{donaldson86}, it admits a compactification by all ideal HYM connections with correct numerical invariants, denoted by $\widehat{M}_{HYM}$, which is first countable and sequentially compact. Therefore, we can take the closure of $M^*_{HYM}$ inside $\widehat{M}_{HYM}$, which gives a compactification $\overline{M}_{HYM}$ by \textit{smoothable} ideal HYM connections. A natural question is whether the complex structure on $M^*_{HYM}$ extends to the boundary of compactified moduli space $\overline{M}_{HYM}$. This was known for integrally polarized projective surface by the fundamental work of \cite{li1993}, and extended to higher dimensions by \cite{GSTW2021}. There is a folklore conjecture that the same analyticity should hold for general compact K\"ahler surfaces \textit{without} assuming any algebracity on both the surface itself or the polarization \cite{BTT2017,BTT2018}, which remains open at this moment.

To connect with this conjecture, consider a flat family of torsion-free $\mu$-semistable sheaves $\mathcal G\to X\times S$ with a fixed determinant. Then there exists a natural comparison map from $S$ to $\widehat{M}_{HYM}$ as follows:
\[
\Psi_S:S\to\widehat{M}_{HYM},\quad s\mapsto[(\gr^S(\mathcal G_s)^{\vee\vee},\mathcal C(\mathcal G_s))].
\]
Here $\gr^S(\mathcal G)^{\vee\vee}$ denotes reflexive hull of the graded object of \textit{some} Seshadri filtration of $\mathcal G$ and $\mathcal C(\mathcal G)$ denotes the algebraic cycle associated with a $\mu$-semistable sheaf; both are well-defined objects, see Section \ref{subsec--basic definitions} for a review.

The first consequence is the following continuity theorem; cf. \cite[Proposition 4.13]{GSTW2021}.
\begin{theorem}\label{thm--continuity theorem}
   Let $S^o$ be the locus of stable locally-free sheaves. Then the image of $\overline{S^o}$ under $\Psi_S$ is contained in $\overline{M}_{HYM}$, and the natural map
   \[
   \Psi_S\vert_{\overline{S^o}}: \overline{S^o}\to\overline{M}_{HYM}
   \]
    is continuous.
\end{theorem}

The second application is the following boundedness result for $\mu$-semistable sheaves with fixed topology:
\begin{theorem}\label{thm-boundedness of semistable sheaves with fixed topology}
    For every fixed rank, $c_1, c_2$, the set of isomorphism classes of $\mu$-semistable torsion-free sheaves with this given topological data is bounded in the sense of \cite{toma16,toma2021}.
\end{theorem}

Therefore, our work here provides further evidence for this analyticity conjecture, and builds up the missing bridge between gauge theory and complex analytic/algebraic geometry; see detailed discussions in Section \ref{sec-discussions}. We were also informed that there is an ongoing project \cite{TM} by A. Teleman and M. Toma on this boundedness result using a different method. The author thanks Professor Toma for email exchange and discussions on related topics.

\

We briefly outline the idea of proof. The key is to understand the convergence of holomorphic sections under the gauge provided by Riemannian convergence theory, a fundamental point appeared in \cite{DS1}. To establish this, we appeal to the robust H\"ormander's weighted $L^2$-theory. However, in contract to \cite{DS1} which only requires a uniform Ricci lower bound, the $L^2$-estimate for holomorphic bundles requires a uniform lower bound on the full curvature tensor, which does not hold for HYM connections developing curvature concentration.

To overcome this issue, we use the bubbling analysis of the degeneration, and construct suitable plurisubharmonic weights at different scales that indeed control the curvature tensor. Due to the curvature concentration, however, those psh weights cannot be uniformly bounded near bubbling points. The crux is then to obtain a uniform control of the blow-up order of those psh weights, which allows us to carry out H\"ormander's estimates for sections with sufficiently high vanishing order along bubble locus. We remark that with extra significant effort, such technique has been extended to non-collapsed polarized cscK surfaces \cite{ZZ26}, where the same issue regarding the lack of a uniform Ricci curvature lower bound appears due to bubbling phenomenon.

Said in another way, we are still able to apply $L^2$-technique uniformly modulo a possibly finite dimensional obstruction. Geometrically, this corresponds to the fact that the Uhlenbeck limit is the normalization of the algebraic limit, which indeed has more local sections. Therefore we simply \textit{define} the algebraic limit as the saturated subsheaf of Uhlenbeck limit, whose local sections are all (sub)-sequential limits of holomorphic sections under Uhlenbeck topology. The proof of Theorem \ref{thm--local convergence in a flat family}, i.e. the desired local algebraic convergence is then a \textit{formal} consequence of the finite dimensionality of the obstruction space. To globalize this, we use the classical theory of Douady on deformation of compact analytic spaces \cite{douady66,douady74}, and its relative version in \cite{pourcin69}. The local convergent data allows us to realize Uhlenbeck convergence inside an infinite dimensional Banach analytic space, so Douady's versality for infinite dimensional parameter spaces would produce a finite dimensional family as required in Theorem \ref{thm-global convergence in a flat family}.

Given Theorem \ref{thm-global convergence in a flat family}, the proof of Theorem \ref{thm--continuity theorem} essentially follows the line of \cite{li1993,GSTW2021}. We use a reduction lemma (cf. Lemma \ref{lem--reduction to common disk}) to obtain two flat families over $\Delta$, which are isomorphic on $\Delta^*$, and have prescribed central fibers: one is the original $\mu$-semistable sheaf and the other is the algebraic limit constructed by Theorem \ref{thm-global convergence in a flat family}. Then the separatedness result in \cite{toma20} extended to slope stability (cf. Theorem \ref{thm--separatedness for slope stability}) identifies the corresponding ideal HYM connection. 

To prove Theorem \ref{thm-boundedness of semistable sheaves with fixed topology}, we use an induction argument to handle possible degeneration to boundaries. According to Theorem \ref{thm-global convergence in a flat family}, any degeneration of HYM connections to Donaldson-Uhlenbeck boundary would be fit into a finite dimensional flat family with correct numerical data, thus belong locally bounded. On the other hand, the double dual of algebraic limit $\mathcal F_\infty$ in Theorem \ref{thm-global convergence in a flat family} is a polystable bundle, which either has smaller energy or must split. In this way one can induct to bound all such degenerate limits, and hence bound all stable bundles with given topological type. By taking double dual and using the Bogomolov's inequality, it is then not hard to bound all semistable torsion-free sheaves.

\subsection*{Notation and Conventions} 
\begin{itemize}
    \item Throughout the paper, $\Psi\left(\epsilon_1, \ldots, \epsilon_k \mid a_1, \ldots, a_l\right)$ denotes a function such that for fixed parameters $a_i$ we have $\lim _{\epsilon_1, \ldots, \epsilon_k \rightarrow 0} \Psi=0$.
\item In a metric space, we use $B_r(p)$ and $B(p,r)$ interchangeably to denote the open metric ball of radius $r$ centered at $p$, and 
\(
A_p(r_0,r_1):=B_{r_1}(p)\setminus \overline{B_{r_0}(p)}
\)
the open annulus centered at $p$ with inner radius $r_0$ and outer radius $r_1$.
\item We use $dV_g$ and $d\mu_g$ interchangeably to denote the volume measure associated with a Riemannian metric $g$. When there is no risk of confusion, we will omit the subscript $g$.
\item We use $\Delta$ to denote the unit disk in $\mathbb C$ and $\Delta^*$ to denote the punctured disk $\Delta\setminus\{0\}$.
\end{itemize}
 
\subsection*{Acknowledgments} 

The author thanks his advisor Song Sun for his consistent support and patient teaching over the years, and Bin Xu for leading him into the field of gauge theory back into time in USTC. The author thanks Xuemiao Chen and Junsheng Zhang for stimulating discussions and suggestions throughout the project. He thanks Anuk Dayaprema for email exchange on discussion of Yang-Mills flows and its relation to Theorem \ref{thm--uhlenbeck compactification is metrizable and compact}, and Matei Toma for his interest in this work and helpful comments on the draft. He also thanks Yohei Hada and Lei You for discussions on the draft. Part of the work was done during the author's visit to the Institute for Advanced Study in Mathematics (IASM) at Zhejiang University in 2025 and 2026 Spring. The author thanks IASM for its warm hospitality. The author was partially supported by the NSF grant DMS-2304692 and the Shoshichi Kobayashi Memorial Fund at UC Berkeley.

\subsection*{Declaration of AI usage} 
The key new $L^2$-estimate for HYMs grew out of the author's attempt to understand the bubbling phenomenon and was developed independently in the fall of 2024. This was also subsequently generalized in a collaborative work for cscK surfaces in the summer of 2025, culminating the results in \cite{ZZ26}. On the other hand, we use GPT-6 Astra to explore the deformation theoretical construction in Section \ref{sec-convergence in a flat family}, especially the delicate theory developed by Douady \cite{douady66,douady74}. The paper was written by the human author. We take full responsibility for the paper's content
and correctness.

\section{Preliminaries}

\subsection{Definitions}\label{subsec--basic definitions}

Let $(X,\omega)$ be a compact K\"ahler manifold.
\begin{definition}
    A torsion-free sheaf $\mathcal{E}$ is said to be $\mu$-semistable with respect to $\omega$ if for any saturated coherent subsheaf $\mathcal{F}\subset\mathcal{E}$ with $0<\mathrm{rank}(\mathcal{F})<\mathrm{rank}(\mathcal{E})$, we have $\mu_\omega(\mathcal{F})\leq\mu_\omega(\mathcal{E})$.
\end{definition}

Since we always fix the underlying K\"ahler manifold, we will omit the $\omega$-notation in any place. 

A semistable sheaf can be filtered by stable sheaves via the so-called Seshadri filtration:
\begin{definition}
    Given a torsion-free $\mu$-semistable sheaf $\mathcal{F}$ on $X$, a filtration of torsion-free sheaves $0=\mathcal{F}_0\subset\mathcal{F}_1\subset...\subset\mathcal{F}_m=\mathcal{F}$ is called a Seshadri filtration if $\mathcal{Q}_j=\mathcal{F}_j/\mathcal{F}_{j-1}$ are all torsion-free and $\mu$-stable of slope $\mu(\mathcal{F})$. The graded object is defined by 
    \[
    \gr^S(\mathcal F)\coloneqq\bigoplus\mathcal Q_j.
    \]
\end{definition}

In general, neither the Seshadri filtraion, nor its graded object is unique \cite{BTT2017}. However, we can still extract some canonical object from this filtration.

Recall that for a general torsion sheaf $\mathcal T$ whose support has codimension $p$, we can define its effective support cycle $\mathcal C_{\mathcal T}$; cf. \cite[Section 2.5.3]{GSTW2021}. Therefore given any $\mu$-semistable sheaf $\mathcal F$ and \textit{some} Seshadri filtration of $\mathcal F$, we define its (codimension $2$) support cycle as
\[
\mathcal C(\mathcal F)\coloneqq\mathcal C_{\gr^S(\mathcal F)^{\vee\vee}/\gr^S(\mathcal F)}\in\mathcal C_{n-2}(X).
\]
The following theorem is then well-known:

\begin{theorem}[Appendix in \cite{BTT2017}]
    For any torsion-free $\omega$-semistable sheaf $\mathcal{F}$, Seshadri filtrations exist. Moreover, the pair $(\mathrm{Gr}(\mathcal{F})^{\vee\vee},\mathcal C(\mathcal F))$ is canonically defined by $\mathcal F$.
\end{theorem}

We also review basic definitions and properties of numerical invariants of coherent sheaves beyond the algebraic setting, following \cite{sibley-wentworth15}. The reader may find references therein for different approaches for construction of these invariants.

Let $X$ be a compact complex manifold. For any coherent sheaf $\mathcal E\to X$, we can associate it with a sequence of cohomology classes, i.e. \textit{Chern characters} $\ch_p(\mathcal E)\in H^{2p}(X;\mathbb Q)$ for $p\geq 0$. One way of construction is to regard $\mathcal E$ as a \textit{real} analytic sheaf, then choose a global resolution by real analytic complex vector bundles. More precisely, let $\mathcal A_X$ be the sheaf of germs of real analytic complex-valued functions on $X$, so we can find a sequence of real analytic complex vector bundles $E_i$ fitting into the following long exact sequence \cite{grauert58}:
\begin{equation}
    0\to E_r\to E_{r-1}\to\cdots\to E_0\to\mathcal E\otimes_{\mathcal O_X}\mathcal A_X\to 0.
\end{equation}
Then we \textit{define} the Chern character of $\mathcal E$ as
\begin{equation}
    \ch_p(\mathcal E)\coloneqq\sum_{i=0}^r(-1)^i\ch_p(E_i).
\end{equation}
It is easy to see this does not depend on the choice of resolutions. It also provides the definition of Chern classes.

The following explains the topological meaning of algebraic multiplicity appeared in support cycle of a torsion sheaf. We essentially use this in the calculation of local algebraic multiplicity in Section \ref{subsec--computing algebraic multiplicity}.
\begin{proposition}[Proposition 3.1 in \cite{sibley-wentworth15}]\label{prop--chern character of torsion sheaf}
    Let $X$ be a compact complex manifold and $\mathcal T\to X$ a torsion sheaf whose support has codimension $p$. Denote by $Z_j\subset\mathrm{supp}\mathcal T$ all irreducible components of codimension exactly $p$, by $[Z_j]$ the corresponding homology classes and $m_j$ the corresponding algebraic multiplicity in $\mathcal C_{\mathcal T}$. Then for all $k<p$, $\ch_k(\mathcal T)=0$, and $\ch_p(\mathcal T)$ is Poincar\'{e} dual to $\sum m_j Z_j$ in $H^{2p}(X;\mathbb Q)$.
\end{proposition}

\subsection{Boundedness and separatedness}

In this subsection, we review the results developed in \cite{toma16,toma20,toma2021}.

We start with the boundedness property. Recall the following definition:
\begin{definition}[Definition 5.1 in \cite{toma2021}]\label{def--boundedness}
    Let $X$ be an analytic space and $\mathfrak C$ be a set of isomorphism classes of coherent sheaves on $X$. We say the set $\mathfrak C$ is bounded if there exists an analytic space $S$, and a coherent sheaf $\mathcal E\to X\times S$, such that for a compact subset $K\subset S$, the set $\mathfrak C$ is contained in the set of isomorphism classes of fibers of $\mathcal E$ over points in $K$.
\end{definition}
This recovers the definition in the classical algebraic geometry, see \cite[Remark 3.3]{toma16}. Note that the flatness of $\mathcal E$ is \textit{not} included in the definition.

Boundedness is preserved under basic operations of sheaves:
\begin{proposition}[Proposition 3.4 in \cite{toma16}]\label{prop--boundedness for kernels, cokernels, images...}
    Let $\mathfrak C,\mathfrak C^{\prime}$ be two bounded sets of coherent sheaves on $X$. Then the following sets are also bounded:
    \begin{itemize}
        \item [(1)] The set of kernels, cokernels and images of sheaf homomorphisms $\mathcal F\to\mathcal F^{\prime}$, when isomorphism classes of $\mathcal F$ and $\mathcal F^{\prime}$ belong to $\mathfrak C$ and $\mathfrak C^{\prime}$ respectively;
        \item [(2)] The set of isomorphism classes of extensions of $\mathcal F$ by $\mathcal F^{\prime}$, for $\mathcal F,\mathcal F^{\prime}$ as above;
        \item [(3)] The set of isomorphism classes of tensor products $\mathcal F\otimes\mathcal F^{\prime}$.
    \end{itemize}
\end{proposition}

Although we will not use Toma's boundedness criterion in this paper, it seems appropriate to include the statement here, so as to illustrate Definition \ref{def--boundedness}.
\begin{definition}
    We say a set $\mathfrak C$ of isomorphism classes of sheaves is dominated, if there exists a bounded set $\mathfrak D$, such that for any element of $\mathfrak C$ is a quotient of an element of $\mathfrak D$.
\end{definition}

Given this definition, the boundedness criterion in \cite{toma16} can be rephrased in the following way:
\begin{theorem}[Main theorem in \cite{toma16}]\label{thm-boundedness criterion}
    Let $(X,\omega)$ be a compact K\"ahler manifold and $\mathfrak C$ a set of isomorphism classes of sheaves on $X$. Then $\mathfrak C$ is bounded if and only if the following holds:
    \begin{itemize}
        \item [(1)] $\mathfrak C$ is dominated;
        \item [(2)] The degrees of sheaves of $\mathfrak C$ are uniformly bounded from above.
    \end{itemize}
\end{theorem}
This extends the celebrated Grothendieck's boundedness criterion in polarized setting \cite[Theorem 2.1]{grothendieck60} to the analytic category. There is also a relative version of this criterion, as further developed in \cite{toma2021}.

\

We will also use the following separatedness result in the compact K\"ahler setting, which is analogous to \cite[Theorem 6.4]{toma20} for slope stability:
\begin{theorem}\label{thm--separatedness for slope stability}
    Let $(X,\omega)$ be a compact K\"ahler manifold and $\mathcal E,\mathcal F$ be two flat families of torsion-free $\mu$-semistable sheaves over the disk $\Delta$. Suppose
    \[
    \mathcal E_t\cong\mathcal F_t,\quad \forall t\in\Delta^*,
    \]
    then we have
    \[
    \gr^S(\mathcal E_0)^{\vee\vee}\cong\gr^S(\mathcal F_0)^{\vee\vee},\quad \mathcal C(\mathcal E_0)=\mathcal C(\mathcal F_0).
    \]
\end{theorem}

A proof of separatedness for Gieseker-Maruyama semistability can be found in \cite{toma20}. However, it is not hard to formulate such a proof for slope semistability by mimicking the argument there. We include such a proof in Appendix \ref{sec--proof of separatedness for slope stability}.

\section{$L^2$-estimate for HYM connections on K\"ahler surfaces}\label{sec-L2 estimate for HYM connections}

In this section, we prove the key analytic ingredient in this paper, namely a uniform $L^2$-estimate result for HYM connections developing singularities. 

To set up, let $B_{10}(p)\subset\subset (X,\omega)$ be a ball with bounded geometry and $\{A_j\}$ be a sequence of HYM connections with Einstein constant $\mu_j$ on $E\to \overline{B_{10}(p)}$ such that
\[
\int_{B_{10}(p)}\vert F_{A_j}\vert^2 d\mu_g\leq\Lambda,\quad \mu\coloneqq\sup_j\vert\mu_j\vert<\infty.
\]
By passing to some subsequence, we may assume $\{A_j\}$ converges in Uhlenbeck sense to an ideal connection $(A_\infty,\mathcal C_\infty)$. Without loss of generality, we may assume $p$ is the unique bubbling point in $\overline{B_4(p)}$, so we can find a sequence of unitary gauges $u_j: E_\infty\vert_{B_{4}(p)\setminus\{p\}}\to E\vert_{B_{4}(p)\setminus\{p\}}$, such that $u_j^* A_j$ converges to $A_\infty$ in $C^{\infty}_{loc}(B_{4}(p)\setminus\{p\})$.

We recall the general pattern of local convergence of holomorphic sections under the Uhlenbeck gauge. Denote by $\mathcal E_j$ the holomorphic structure on $(E,h)$ defined by $A_j$. Let $s_j\in\Gamma(B_2(p),\mathcal E_j)$ be a sequence of holomorphic sections. Since they are $L^2$ over $B_1(p)$, we may normalize them and assume $\Vert s_j\Vert_{L^2(B_1(p))}=1$. Then the standard Bochner formula and Moser iteration imply that
\[
\sup_{B_r(p)}\vert s_j\vert\leq C(r,\mu,g)\Vert s_j\Vert_{L^2(B_1(p))}=C.
\]
Therefore by passing to some subsequence, we may assume $u_j^* s_j$ converges to a holomorphic section $s_\infty$ in $C^{\infty}_{loc}(B_1(p)\setminus\{p\})$. By Hartogs' extension, $s_\infty$ extends across $p$ as an element in $\Gamma(B_1(p),\mathcal E_\infty)$. From now on we will use the terminology that $\{s_j\}$ $C^{\infty}_{loc}$-converges to $s_\infty$ in the Uhlenbeck sense, without explicitly pointing out the unitary gauge if unnecessary.

By Fatou's lemma, we know that $\Vert s_\infty\Vert_{L^2(B_1(p))}\leq 1$. The fundamental difficulty in understanding the \textit{local} convergence of holomorphic sections, however, lies in the potential $L^2$-concentration near the boundary $\partial B_1(p)$. That is, we can only conclude that $s_j$ weakly converges to $s_\infty$ in $L^2$-sense, but in general the $L^2$-norm could be lost when passing to the limit. In fact, \textit{a prior} more severe thing could happen: it is possible that all (sub-)sequences of sections normalized in this way converge to $0$, so we lose control of the holomorphic geometry for Uhlenbeck convergence. This reveals the key distinction between global geometry and local geometry \cite{li1993,GSTW2021}. When the sequence is given by rescaling of a fixed (singular) connection, it is also the fundamental issue handled in \cite{chen-sun2021}.

In this paper, we will address this question for \textit{arbitrary} convergent sequences in complex dimension two via the robust H\"ormander's weighted $L^2$-technique. The crux is the following construction of psh weights:
\begin{theorem}\label{thm-existence of psh weights}
    By passing to some subsequence, we can find a sequence of continuous psh functions $\varphi_j$ on $B_2(p)$, such that for some uniform constants $N,C>0$ depending only on $\Lambda,\mu$ and the local geometry of $B_2(p)$, the following is true:
    \begin{itemize}
        \item [(1)] $(\mathcal E_j,h\cdot e^{-2\varphi_j})\geq_{\mathrm{Nak}}0$ on $B_1(p)$;
        \item [(2)] There exist finitely many points $p_{j,I}$ converging to $p$ as $j\to\infty$, such that
        \[
        N\sum_I \log d(\cdot,p_{j,I})-C\leq\varphi_j(\cdot)\leq 0\ \text{ on}\ B_1(p)
        \]
    \end{itemize}
\end{theorem}

The proof is essentially the same as \cite[Section 5]{ZZ26} by grafting psh weights from the bubble tree. However, since we fix the underlying surface in the HYM setting, it is worth spelling out details more explicitly. As we will see, as the base is $C^{\infty}$-close to Euclidean space in arbitrary scales, the construction of psh weights is straightforward \textit{without} an extra need of the gluing procedure.

Given Theorem \ref{thm-existence of psh weights}, by cutting off in the tangent cone scale, one can immediately conclude that sections of $\mathcal E_\infty$ with sufficiently high vanishing order must arise as (sub-)sequential limits; cf. \cite[Proposition 8.2]{ZZ26}.
\begin{corollary}\label{cor--L2-convergence for sections with high vanishing order}
    There exists an integer $m$ depending only on $\Lambda,\sup\vert\mu_j\vert$ and the local geometry of $B_2(p)$, such that for any holomorphic section $s_\infty\in\Gamma(B_2(p),\mathcal I^{m}_p\mathcal E_\infty)$, there exists a sequence of holomorphic sections $s_{j_k}\in\Gamma(B_1(p),\mathcal E_{j_k})$ converging to $s_\infty$ in the Uhlenbeck sense both in $C^{\infty}_{loc}(B_1(p)\setminus\{p\})$ and $L^2(B_1(p))$.
\end{corollary}

This section is organized as follows. In Section \ref{subsec--bubbling of yang-mills connections}, we review the standard bubbling analysis for HYM connections in real dimension four, especially the neck region estimate. Then in Section \ref{subsec--construction of psh weights}, we prove Theorem \ref{thm-existence of psh weights}. We also include an explicit example calculation in Section \ref{subsec-an example}, illustrating that our estimate is optimal.

\subsection{Bubbling of HYM connections}\label{subsec--bubbling of yang-mills connections}

In this subsection, we review the standard theory on bubbling of HYM connections in complex dimension two.

To improve clarity and ease the notation, we need the notion of a rooted tree.
Let $\mathcal{T}$ be a rooted tree with a finite number of vertices. The vertices of $\mathcal{T}$ are partially ordered such that $I \leq  J$ if and only if the unique path from the root to $J$ passes through $I$. 
If $I\leq J$ and $I\neq J$, we denote this by $I<J$. The depth of a vertex, denoted by  $\mathrm{dep}(I)$ is defined to be the length of the path to its root.  In particular, the root has depth zero. The depth of a tree is the maximum depth of any vertex.
For any vertex $I$, we let $C(I)$ denote its set of children; if it is not the root, we let $P(I)$ denote its unique parent.

Using this notation, we can state the following classical bubble tree decomposition in complex dimension two as follows:
\begin{theorem}\label{thm--bubble tree decomposition}
   Passing to a subsequence, there exist a rooted tree with finitely many vertices \( \mathcal{T} \) and a constant \( \eta > 0 \) such that, for each \( j \) and each vertex \( I \in \mathcal{T} \), there exist points \( p_{j,I} \in X \) with \( p_{j,I} \to p_{\infty} \), scaling factors \( \lambda_{j,I} \to \infty \) as \( j \to \infty \), and constants \( \delta_I \in (0,1) \) satisfying the following properties:
    \begin{itemize}
    \item for the root vertex $I$,  we denote $p_{j,I}=p_{j,1}$ and we adopt the convention that $\lambda_{j,P(1)}=1$, then for any vertex $I\in \mathcal{T}$, we have 
    \begin{equation}\label{curvature estimate on annulus}
        \int_{A_{p_{j,I}}(\lambda_{j,I}^{-1},\delta_{I}\lambda_{j,P(I)}^{-1})}|F_{A_j}|^2 d\mu_g\leq \epsilon_0,
    \end{equation}
        \item for any $I<J$, we have $p_{j,J}\in B(p_{j,I}, \lambda_{j,I}^{-1})$ and $\lim_{j\rightarrow \infty} \frac{\lambda_{j,I}}{\lambda_{j,J}}=0,$ 
        \item for any $I<J$, we have either $p_{j,J}=p_{j,I}$ or $\lim_{j\rightarrow \infty}\lambda_{j,I}d(p_{j,J},p_{j,I})\geq \eta>0$,
        \item for any $I$,  $(A_j, p_{j,I},\lambda_{j,I} ^2g)$ converges to an ideal HYM connection $Z_{I}=(B_I,\mathcal C_I)$ on $(T_p X,g_0)$ in the pointed Uhlenbeck sense. Moreover, either $B_I$ is non-flat, or $\#\vert\mathcal C_I\vert\geq 2$.
        \item if a vertex $I$ has no children, then $(A_j, p_{j,I},\lambda_{j,I} ^2g)$ converges smoothly to a smooth HYM connection $B_{I}$ in the pointed Uhlenbeck sense.
        \item for any $I$, $C(I)$ has a natural one-to-one correspondence to the bubbling locus $\vert\mathcal C_I\vert$. Moreover, if $J\in C(I)$ corresponds to a point $q_J\in \vert\mathcal C_I\vert$, then the asymptotic cone of $Z_J$ is isomorphic to the local tangent cone of $Z_I$ at $q_J$, i.e.
        \[
        \frac{1}{8\pi^2}\int_{T_p X}\vert F_{B_J}\vert^2d\mu_{g_0}+\sum_{q\in\vert\mathcal C_{ J}\vert}m_q=m_{q_J}.
        \]
        Here $m_q, m_{q_J}$ denote the analytic multiplicity of $q, q_J$ appearing in $\mathcal C_J,\mathcal C_I$.
    \end{itemize}
\end{theorem}

 For the proof of Theorem \ref{thm--bubble tree decomposition}, one might proceed using the energy scale method in \cite{bando1990a,bando1990b}; see also \cite{ZZ26} for a detailed account on cscK surfaces, which applies verbatim here.

We also need the following neck region estimate:
\begin{theorem}\label{thm--neck region estimate}
    Under the setting of Theorem \ref{thm--bubble tree decomposition}, there exist uniform constants $K_0,C,j_0>1$ and $\delta>0$, such that for all $j\geq j_0$ and any vertex $I$, we have
    \[
    r^2\vert F_{A_j}\vert(q)\leq C\max\{(\lambda_{j,I}r)^{-\delta},(\lambda_{j,P(I)}r)^{\delta}\},
    \]
    where $q\in A_{p_{j,I}}(K_0\lambda_{j,I}^{-1},K_0^{-1}\lambda_{j,P(I)}^{-1})$ and $r=d(q,p_{j,I})$.
\end{theorem}
For the proof of Theorem \ref{thm--neck region estimate}, we refer to \cite{donaldson83,freed-uhlenbeck,DK97}. Note that the optimal exponent $\delta=2$ was obtained in \cite{groisser-parker97}, although it is not needed for our purposes. One quick proof is to use the refined Kato's inequality for Yang-Mills connections in \cite{rade93} (see also \cite[Lemma 3.1]{groisser-parker97}), using the fact that $(X,\omega)$ is K\"ahler, and then immediately appeal to the proof of \cite[Proposition 3]{bando1990a}.

\subsection{Construction of plurisubharmonic weights}\label{subsec--construction of psh weights}

In this subsection, we construct suitable psh weights which control the curvature blow-up. We start with an explicit model psh function:
\begin{lemma}\label{lem--model psh weight}
    Given any $0<\epsilon\leq 2$, we have the following:
    \begin{itemize}
        \item [(1)] Let $\phi(z;\epsilon)=\log(1+(1+\vert z\vert^2)^{\epsilon})\in C^{\infty}(\mathbb{C}^n)$, then the smallest eigenvalue $\lambda_1(z;\epsilon)$ of $\sqrt{-1}\partial\pp\phi(z;\epsilon)$ satisfies:
\[
  \lambda_1(z;\epsilon)\geq\frac{\epsilon^2}{2(1+\vert z\vert^2)(1+(1+\vert z\vert^2)^{\epsilon})}.
\]
      Therefore, $\phi(\cdot)=\phi(\cdot;\epsilon)$ is a strictly psh function on $\mathbb{C}^n$ satisfying $\phi\sim\log r$ and $\sqrt{-1}\partial\pp\phi\gtrsim \vert z\vert^{-2-2\epsilon}\omega_0$ for $\vert z\vert\gtrsim 1$.

     Moreover, for all $\lambda$ large, define $\hat\phi_{\lambda}(z)\coloneqq\phi(\lambda z)$, then 
     \begin{equation}\label{equa--positivity for model psh at infinity}
         \sqrt{-1}\partial\pp\hat\phi_{\lambda}\gtrsim
         \begin{cases}
             \lambda^2\omega_0,& \text{if}\ \vert z\vert\lesssim\lambda^{-1},\\
             \lambda^{-2\epsilon}\vert z\vert^{-2-2\epsilon}\omega_0,& \text{if}\ \vert z\vert\gtrsim\lambda^{-1}.
         \end{cases}
     \end{equation}
    If we normalize by $\phi_{\lambda}=\hat\phi_{\lambda}-\sup_{B_1}\hat\phi_{\lambda}$, then 
    \begin{equation}\label{equa--growth estimate for model psh at infinity}
        2\epsilon\log\vert z\vert\leq\phi_{\lambda}\leq 0,\quad\text{if}\ 0<\vert z\vert< 1.
    \end{equation}
    \item [(2)] 
    Similarly, consider the rescaled psh function with parameter
    \[
    \psi_\lambda(z;\epsilon)=\log(1+(\lambda\vert z\vert)^\epsilon)-\log(1+\lambda^\epsilon)\in C^{\infty}(\mathbb C^n\setminus\{0\})\cap C^0(\mathbb C^n).
    \]
    It satisfies the uniform estimate
    \begin{equation}\label{equa--growth estimate for model psh near zero}
        \epsilon\log\vert z\vert\leq\psi_\lambda\leq 0\quad\text{if}\ 0<\vert z\vert< 1,
    \end{equation}
    and
    \begin{equation}\label{equa--positivity for model psh near zero}
        \sqrt{-1}\partial\pp\psi_\lambda\gtrsim\lambda^\epsilon\vert z\vert^{-2+\epsilon}\omega_0\quad\text{if}\ 0<\vert z\vert\lesssim\lambda^{-1}.
    \end{equation}
   
    \end{itemize}
\end{lemma}
\begin{proof}
    Brute force computation.
\end{proof}

\begin{remark}
    A sharper positivity $O(r^{-2}\vert\log r\vert^{-1-\alpha})$ can be obtained using the weights in \cite[Lemma 3.7]{sun-zhang2023} and \cite[Section 4]{zhang2025}, although it is not needed for our purpose.
\end{remark}

Now we construct the desired psh weights:
\begin{proof}[Proof of Theorem \ref{thm-existence of psh weights}]
    According to Theorem \ref{thm--bubble tree decomposition}, we can find constants $\delta=\delta_1\in(0,1)$ and $K>K_0^2$ in Theorem \ref{thm--neck region estimate}, so that for any fixed
    \[
    0<\rho<\min\{\delta_1/4, K^{-\frac12}/4\},
    \]
    the ball $B_{\rho}(p)$ can be decomposed into finitely many pieces
    \[
    B_\rho(p)=\big(\bigcup_{ I\in\mathrm{V}(\mathcal T)}\mathcal B_I\big)\cup\big(\bigcup_{ I\in\mathrm{V}(\mathcal T)}\mathcal A_I\big),
    \]
    where
    \[
    \mathcal B_I= B(p_{j,I},K\lambda_{j,I}^{-1})\setminus\bigcup_{J\in C(I)} B(p_{j,J},K^{-1}\lambda_{j,I}^{-1})
    \]
    is the bubble region, and
    \[
    \mathcal A_I= A_{p_{j,I}}(K^{\frac12}\lambda_{j,I}^{-1},K^{-\frac12}\lambda_{j,P(I)}^{-1})
    \]
    is the neck region. The formulation of quantitative decomposition here is taken from \cite{NV19}. Note that we choose $K$ large enough such that estimates in Lemma \ref{lem--model psh weight} hold, and then fix such a choice. 
    
    Since our base surface is fixed, we may assume there exists a holomorphic coordinate $(z_1,z_2)$ near $p$ such that
    \[
    z_i(p)=0,\quad \omega=\omega_0+O(\vert z\vert^2).
    \]
    It suffices to construct psh weights under this explicit coordinate, which indeed control the curvature along different regions. This is where we can avoid the general complicated gluing procedure in \cite[Section 5]{ZZ26}. Denote by $\mathbb C^2\ni z_{j,I}\to 0$ the coordinate of bubbling center $p_{j,I}$.

    Let $\mathcal B_I$ be a bubble region. The geometry of $A_j$ is bounded over $\mathcal B_I$ under the rescaled metric $\lambda_{j,I}^2g$, so we may assume for all $j$ large, we have
    \[
    \sup_{\mathcal B_I}\vert F_{A_j}\vert_g\leq C_I\lambda_{j,I}^2.
    \]
    Since $\mathcal B_I\subset\{z:\vert z-z_I\vert\leq 2K\lambda_{j,I}^{-1}\}$, we can consider $\varphi_{j,I,0}=N_I\phi_{\lambda_{j,I}}(z-z_I)$ on $B(p,1/2)$ for some $N_I$ large, using the rescaled weight in Lemma \ref{lem--model psh weight}. The estimate \eqref{equa--positivity for model psh at infinity} shows that this psh function controls the curvature, and estimate \eqref{equa--growth estimate for model psh at infinity} says it has a controlled logarithmic blow-up order.

    Then we deal with the neck region. Let $\mathcal A_I$ be a neck region, for which we may assume
    \[
    \mathcal A_I\subset\{z:\frac{1}{2}K^{\frac12}\lambda_{j,I}^{-1}\leq\vert z-z_I\vert\leq 2K^{-\frac{1}{2}}\lambda_{j,P(I)}^{-1}\}.
    \]
    Given Theorem \ref{thm--neck region estimate}, we may decompose further this region into two parts with uniform estimates:
    \begin{equation}\label{equa--curvature estimate along neck under coordinate}
         \vert F_{A_j}\vert_g\leq
    \begin{cases}
        C\lambda_{j,I}^{-\delta}\vert z-z_I\vert^{-2-\delta},&\text{if}\ K^{\frac{1}{2}}\lambda_{j,I}^{-1}\lesssim\vert z-z_I\vert\lesssim\lambda_{j,I}^{-\frac{1}{2}}\lambda_{j,P(I)}^{-\frac12},\\
        C\lambda_{j,P(I)}^{\delta}\vert z-z_I\vert^{-2+\delta},&\text{if}\ \lambda_{j,I}^{-\frac{1}{2}}\lambda_{j,P(I)}^{-\frac12}\lesssim\vert z-z_I\vert\lesssim K^{-\frac{1}{2}}\lambda_{j,P(I)}^{-1}.
    \end{cases}
    \end{equation}

For the inner annulus, we consider $\varphi_{j,I,1}=N_I\phi_{\lambda_{j,I}}(z-z_I)$ with $2\epsilon<\delta$. By \eqref{equa--positivity for model psh at infinity} and \eqref{equa--curvature estimate along neck under coordinate}, one can verify directly that this weight controls the curvature, and \eqref{equa--growth estimate for model psh at infinity} provides the blow-up order control. On the other hand, for the outer annulus, we consider $\varphi_{j,I,2}=N_I\psi_{\lambda_{j,P(I)}}(z-z_I)$ with $0<\epsilon\leq\delta$, which by \eqref{equa--growth estimate for model psh near zero} and \eqref{equa--positivity for model psh near zero} is a desired psh weight.

Finally, since the curvature is uniformly bounded on $B_1(p)\setminus B_{\rho/2}(p)$, we can choose $\varphi_0= N\vert z\vert^2-N$ on $B_1(p)$ to control the curvature over $A_p(\rho/2,1)$. Then
\[
\varphi_j\coloneqq\varphi_0+\sum_{I\in\mathrm{V}(\mathcal T),k\in\{0,1,2\}}\varphi_{j,I,k}
\]
is a desired psh function.
   
\end{proof}

\subsection{An example}\label{subsec-an example}

In this subsection, we include an explicit example, which shows that in general we cannot expect all sections of $\mathcal E_\infty$ are sequential limits of sections of $\mathcal E_j$ under Uhlenbeck gauge. Therefore, Corollary \ref{cor--L2-convergence for sections with high vanishing order} is optimal in this sense.

Let $A$ be the basic instanton on $\mathbb{R}^4$, i.e. it is of rank $2$ and extends as a smooth instanton on a bundle with $c_2=1$ over $S^4$. By scaling down $A$, we obtain a family of instantons $A_\lambda$, which converges
 to a smooth connection $A_\infty$ away from the origin as $\lambda\to 0$. The origin is the unique bubbling point, which has multiplicity one. This is the baby model of HYM connections developing curvature singularities.
 
By Uhlenbeck's removable singularity theorem, we can choose a unitary gauge transform $u$ over $B_1-\{0\}$, such that the limit connection $ u^*A_{\infty}$ extends to 
a smooth connection over $B_1$. By energy condition, we know that it has to be the trivial connection. In particular, the holomorphic vector bundle defined by $u^* A_\infty$ is holomorphically trivial. We prove:
 \begin{proposition}\label{prop-explicit example for basic instanton}
Let $\{s_{\lambda_k}\}$ be any sequence of holomorphic sections converging to $s_{\infty}$ in the Uhlenbeck sense.
 Then $u^*s_{\infty}\in\mathcal{O}\oplus\mathcal{I}$, where $\mathcal{I}$ is the ideal generated by $x$ and $y$. Conversely, any holomorphic
 section $s\in\mathcal{O}\oplus\mathcal{I}$ is a sequential limit of holomorphic sections under Uhlenbeck gauge.
 \end{proposition}
\begin{proof}
    
We use the cohomological description of instantons, as obtained by Donaldson in \cite{donaldson84}; see also \cite[Chapter 2]{nakajima99} for an exposition. Here we use $(x,y)$ to denote the coordinate of $\mathbb{C}^2$. The holomorphic bundle $\mathcal E_\lambda$ defined by $A_\lambda$ is given by the cohomology sheaf of the following sequence, known as a monad:
\[
 \underline{\mathbb{C}}\xrightarrow{\alpha=(x,y,\lambda,0)}\underline{\mathbb{C}}^4\xrightarrow{\beta=(-y,x,0,\lambda)}\underline{\mathbb{C}}.
\]
We equip these trivial bundles with trivial Hermitian metrics, so $\mathcal{E}_{\lambda}$ is endowed
 with a Hermitian metric under the smooth identification
 \[
  \mathcal{E}_{\lambda}\cong E_{\lambda}:=\mathrm{Ker}\ \alpha^H\cap \mathrm{Ker}\ \beta. 
 \]
  The Chern connection gives the desired instanton $A_\lambda$.

By straightforward computation, we see that
\[
 \alpha^H=(\bar{x},\bar{y},\lambda,0), 
\]
  so a smooth basis for $E_{\lambda}$ is given by
\[
E_1=(1,0,-\lambda^{-1}\bar{x},\lambda^{-1}y), E_2=(0,1,-\lambda^{-1}\bar{y},-\lambda^{-1}x). 
\]
   Therefore by Gram-Schmidt process, a unitary basis is given by
\[
\begin{aligned}
   & e_1=\frac{1}{(\vert x\vert^2+\vert y\vert^2+\lambda^2)^{\frac12}}(\lambda,0,-\bar{x},y),\\
   & e_2=\frac{1}{(\vert x\vert^2+\vert y\vert^2+\lambda^2)^{\frac12}}(0,\lambda,-\bar{y},-x),
\end{aligned}
\]
under which we can identify $E_{\lambda}$ with the trivial $SU(2)$-bundle $\underline{\mathbb{C}}^2$ over $B_1$. One can also compute the connection matrix $A_\lambda$ using this unitary gauge. To see this, put $r=\sqrt{\vert x\vert^2+\vert y\vert^2}$, and consider the gauge transform
\[
    u:B_1\setminus \{0\}\to SU(2),(x,y)\mapsto 
\frac1r\left(\begin{matrix}
     \bar x &\bar y\\
     -y & x
\end{matrix}\right).
\]
Then one can find that
\[
A_\lambda=\frac{r^2}{r^2+\lambda^2} u^{-1}du.
\]
In particular, it converges to $A_\infty=u^{-1}du$ in $C^{\infty}_{loc}(B_1\setminus\{0\})$. We also identify the gauge transform for removing singularity at $0$:
\[
    \widehat{A}_\infty=uA_\infty u^{-1}-(du) u^{-1}=0.
\]

Now we choose a holomorphic basis of $\mathcal{E}_{\lambda}$ over $B_1$. It can be represented by two linearly
 independent holomorphic sections of $\mathrm{Ker}\ \beta$ which are not contained in $\mathrm{Im}\ \alpha$. Explicitly, we can choose
\[
 s_1=(1,0,0,\lambda^{-1}y),\quad s_2=(0,1,0,-\lambda^{-1}x),
\]
    and thus $\{\bar{s}_1,\bar{s}_2\}$ forms a holomorphic basis of $\mathcal{E}_{\lambda}$.

We compute the coordinate of $\bar s_i$ under unitary gauge. Using the smooth identification $\mathcal{E}_{\lambda}\cong E_{\lambda}$, these two holomorphic
sections are given by:
\[
s_i^{\prime}=s_i-\alpha(\alpha^{H}\alpha)^{-1}\alpha^{H}s_i,
\]
from which we obtain
\[
\begin{aligned}
    & s_1^{\prime}=\frac{1}{\lambda(\vert x\vert^2+\vert y\vert^2+\lambda^2)^{\frac12}}((\vert y\vert^2+\lambda^2)e_1-\bar{x}ye_2),\\
    &  s_2^{\prime}=\frac{1}{\lambda(\vert x\vert^2+\vert y\vert^2+\lambda^2)^{\frac12}}(-x\bar{y}e_1+(\vert x\vert^2+\lambda^2)e_2).
\end{aligned}
\]
   By abusing notation, we will just them by $s_i$. Using unitary gauge, we can compute its $L^2$-norm via
\[
\vert s_1\vert^2=\frac{\vert y\vert^2+\lambda^2}{\vert x\vert^2+\vert y\vert^2+\lambda^2}+\frac{\vert y\vert^2}{\lambda^2}.
\]
    Thus $\Vert s_1\Vert_{L^2(B_1)}\sim\lambda^{-1}$, and we shall consider the normalized section
\[
    \hat{s}_1=\lambda s_1= \frac{1}{{(\vert x\vert^2+\vert y\vert^2+\lambda^2)}^{\frac12}}(\vert y\vert^2+\lambda^2,-\bar{x}y)
\]
As a result, $\hat{s}_{1,\lambda}$ $C^{\infty}_{\text{loc}}$-converges over $B_1-\{0\}$ to
\[
    \hat{s}_{1,\infty}=\frac{1}{(\vert x\vert^2+\vert y\vert^2)^{\frac12}}(\vert y\vert^2,-\bar{x}y).
\]
We can do similar computation, so $\hat{s}_{2,\lambda}$ $C^{\infty}_{\text{loc}}$-converges over $B_1-\{0\}$ to:
\[
    \hat{s}_{2,\infty}=\frac{1}{(\vert x\vert^2+\vert y\vert^2)^{\frac12}}(-x\bar{y},\vert x\vert^2).
\]
Therefore, under the singularity-removing gauge $u$, we have
\[
    u(\hat{s}_{1,\infty})=(0,-y),\quad u(\hat{s}_{2,\infty})=(0,x).
\]

Suppose now $s_k$ is any sequence of holomorphic sections of $\mathcal E_{\lambda_k}$ converging to $s_{\infty}$ in the Uhlenbeck sense. We write $s_{k}=f_k^1\hat{s}_{1,\lambda_k}+f_k^2\hat{s}_{2,\lambda_k}$.  Then under unitary gauge, we have
\[
    u(s_{k})=
\left(\begin{matrix}
     \frac{\lambda_k^2}{r\sqrt{r^2+\lambda_k^2}}(\bar{x}f_k^1+\bar{y}f_k^2)\\
     \frac{\sqrt{r^2+\lambda_k^2}}{r}(-yf_k^1+xf_k^2)
\end{matrix}\right).
\]
Our assumption is that both components converges in $C^{\infty}_{loc}(B_1\setminus\{0\})$. By restricting to a one-dimensional slice and applying Cauchy's integral formula, the second component must converge to a holomorphic function vanishing at the origin.
Therefore $u(s_{\infty})\in\mathcal{O}\oplus\mathcal{I}$.

Conversely, for any $s=(f,0)\in\mathcal{O}\oplus\mathcal{I}$, the sequence $s_{\lambda}=\lambda^{-2}f\cdot(x\hat{s}_{1,\lambda}+y\hat{s}_{2,\lambda})$
converges to $s$. On the other hand, for any $s=(0,f)\in \mathcal{O}\oplus\mathcal{I}$, let $f=-yf_1+xf_2$, then the 
sequence $s_{\lambda}=f_1\hat{s}_{1,\lambda}+f_2\hat{s}_{2,\lambda}$ converges to $s$. This finishes our proof.
\end{proof}

Let us define the sheaf $\mathcal F_\infty$ as a subsheaf of $\mathcal E_\infty$ generated by all limit sections, then the above calculation can be rephrased as the fact that $\mathcal F_\infty\cong\mathcal O\oplus (x,y)$. As we will see in Section \ref{sec-convergence in a flat family}, this is \textit{not} simply a coincidence. The length of $\mathcal E_\infty/\mathcal F_\infty$ at $0$ must be equal to the analytic multiplicity, which is one in this model case. This forces $\mathcal F_\infty$ to be isomorphic to $\mathcal O\oplus(x,y)$. The above calculation actually motivates the Definition \ref{def--algebraic limit} as a local analogue of algebraic limits in \cite{li1993,GSTW2021}, and finally leads to the construction in Section \ref{sec-convergence in a flat family}.

\section{Convergence inside an analytic flat family}\label{sec-convergence in a flat family}

In this section, we prove Theorem \ref{thm--local convergence in a flat family} and Theorem \ref{thm-global convergence in a flat family}.

\subsection{Local construction}

In this subsection, we prove Theorem \ref{thm--local convergence in a flat family}. The Hilbert space method is inspired by \cite{li20}. By choosing a sufficiently small coordinate ball $U$, we may assume $\mathcal E_\infty$ is holomorphically trivial over $U$ and $0\in U$ is the unique bubbling point.

\subsubsection{Hilbert space method}\label{subsubsec--Hilbert space method}
We first introduce the definition of the algebraic limit.
\begin{definition}\label{def--algebraic limit}
    Let $\sigma$ be a subsequence of the original sequence. We define $M(\sigma)$ as a subset of $\mathcal E_{\infty,0}$ consisting of germs satisfying the following: there exists a representative $s$ on a ball $D\ni 0$, such that for any subsequence $\tau\subset\sigma$, there exists a further subsequence $\tau^{\prime}\subset\tau$ and sections
    \[
    s_j\in\Gamma(D,\mathcal E_j),\quad j\in\tau^{\prime},\quad s_j\to s\ \text{in} \ C^{\infty}_{loc}(D\setminus\{0\}).
    \]
\end{definition}
Note that $D$ depends on the representative $s$, but not on the subsequence $\tau$. We always allow possibly shrinking the domain $D$. This \textit{``sub-subsequence''} definition makes the module structure legitimate, i.e.
\begin{lemma}
    $M(\sigma)$ is an $\mathcal O_0$-submodule of $\mathcal E_{\infty,0}$. If $\tau\subset\sigma$, then $M(\sigma)\subset M(\tau)$.
\end{lemma}
\begin{proof}
    Direct check of the definition.
\end{proof}

Now we apply the crucial $L^2$-estimate. By Corollary \ref{cor--L2-convergence for sections with high vanishing order}, there exists an integer $m>0$ such that
\[
\mathfrak m_0^m\mathcal E_{\infty,0}\subset M(\sigma)\subset\mathcal E_{\infty,0}.
\]
It follows that
\[
0\leq\ell(\sigma)\coloneqq\length(\mathcal E_{\infty,0}/M(\sigma))<\infty.
\]
Thus we can saturate the choice of $M(\sigma)$ as follows:
\begin{lemma}\label{lem--saturation of algebraic limit}
    There exists a subsequence $\sigma$, such that $M(\tau)=M(\sigma)$ for any subsequence $\tau\subset\sigma$. For this choice, if $s$ is a Uhlenbeck limit of sections $s_j\in\Gamma(D,\mathcal E_j)$ along some subsequence $\tau\subset\sigma$, i.e. $j\in\tau$, over some ball $D\ni 0$, then $s\in M(\sigma)$.
\end{lemma}
\begin{proof}
    As we discussed above, all colengths $\ell(\sigma)$ is finite. The set of non-negative integers 
    \[
    \{l(\sigma):\sigma \ \text{is a subsequence of the original sequence}\}
    \]
    admits a minimizer $\sigma$. If $\tau\subset\sigma$, then the inclusion $M(\sigma)\subset M(\tau)$ implies that 
    \[
    \ell(\tau)\leq\ell(\sigma).
    \]
    The mininality of $\ell(\sigma)$ implies that $\ell(\tau)=\ell(\sigma)$, thus $M(\tau)=M(\sigma)$.

    Now let $s$ be a Uhlenbeck limit along a subsequence $\tau\subset\sigma$. Every further subsequence of $\tau$ has the same convergence, so $s\in M(\tau)$, thus $s\in M(\sigma)$.
    
\end{proof}

Now we simply relabel this sequence by $\{j\}$, and set $M=M(\sigma)$ and $Q_0=\mathcal E_{\infty,0}/M$. Since $\mathfrak m_0^m Q_0=0$, we can regard it as a coherent sheaf on the non-reduced point space $(0,\mathcal O_0/\mathfrak m_0^m)$, so pushforward by the natural inclusion extends it as a coherent sheaf $\mathcal Q_0$ on $U$. There is also a natural evaluation morphism $\mathcal E_\infty\to\mathcal Q_0$. Now define
\[
\mathcal F_\infty\coloneqq\ker\left(\mathcal E_\infty\to\mathcal Q_0\right),
\]
which is coherent, has stalk $M$ at $0$ and equals $\mathcal E_\infty$ on $U\setminus\{0\}$. In particular, it is torsion-free and $\mathcal F_\infty^{\vee\vee}=\mathcal E_\infty$.

By coherence we may choose a presentation over a possibly smaller ball $B^+$
\[
    0\to\mathcal O_{B^+}^k\xrightarrow{R_\infty}\mathcal O_{B^+}^N\xrightarrow{q_\infty}\mathcal F_\infty\to 0.
    \]
    Note that here we can assume the kernel is locally-free, since $\mathcal F_\infty$ is torsion-free and $\dim_{\mathbb C} X=2$. Moreover, we can assume columns of $q_\infty$ are Uhlenbeck limits of sections over $B^{+}$ along a common sub-sequence. We may relabel again to denote this sub-sequence by $\{j\}$, which does not change the module $M$ we construct. Thus the approximating sequence of sections defines a sequence of morphisms
    \[
    \mathcal O_{B^+}^N\xrightarrow{q_j}\mathcal E_j,
    \]
    such that columns of $q_j$ converges to columns of $q_\infty$ in Uhlenbeck sense. It is not hard to see that $q_j$ are generically surjective for $j$ large. Our goal is to show $q_j$ is indeed a surjection for all $j$ large, by possibly shrinking the ball.

    Fix $B\subset\subset B^+$. Then by standard Bochner technique and Moser iteration, we can assume
\begin{equation}\label{equa--uniform interior estimate for q_j}
    \sup_j\sup_{\bar B}\vert q_j\vert\leq C
\end{equation}
for some uniform constant $C>0$. Here $\vert q_j\vert$ denotes the pointwise norm induced by standard Hermitian metric on $\mathcal O_{B^+}^N$ and the given Hermitian metric $h$ on $E$.

Consider the Bergman space
\[
\mathcal D\coloneqq A^2(B,\mathbb C^N),\quad H_\infty\coloneqq A^2(B,\mathcal E_{\infty})
\]
i.e. vector-valued $L^2$-integrable holomorphic functions on $B$. Recall we fix a holomorphic trivilization of $\mathcal E_\infty$ over $U$ and the HYM metric is smooth, elements of $H_\infty$ are also identified with vector-valued $L^2$-holomorphic functions with an equivalent norm. By Cauchy's interior estimate, we know that both $\mathcal D$ and $H_\infty$ are complete, hence are Hilbert spaces. We also consider
\[
V\coloneqq\{f\in H_\infty: f_0\in M\}.
\]
Then we have:
\begin{lemma}
    $V$ is closed in $H_\infty$, hence being complete.
\end{lemma}
\begin{proof}
    For any $f\in H_\infty$, let $p_f$ be its truncated Taylor expansion of order $\leq m$ under the given holomorphic frame. Then by Corollary \ref{cor--L2-convergence for sections with high vanishing order}, 
    \[
    f_0-(p_f)_0\in M.
    \]
    Here the subscript $0$ denotes the germ of section. It follows that $f_0\in M$ if and only if $(p_f)_0\in M$. Therefore we define the space of allowed truncated Taylor polynomials
    \[
    \mathcal J\coloneqq\{p\in\mathcal P_{\leq m}(\mathbb C^2,\mathbb C^r):p_0\in M\},
    \]
    which is a linear subspace of a finite dimensional space, hence being closed. The Taylor truncation is a continuous map by Cauchy's estimate, i.e.
    \[
    \mathcal T: H_\infty\to\mathcal P_{\leq m},\quad f\mapsto p_f
    \]
    is continuous. Thus $V=\mathcal T^{-1}(J)$ is the preimage of a closed subset, hence being closed.
\end{proof}

Now we switch from the sheaf theoretical perspective to the operator theoretical perspective. Denote by $Q$ the holomorphic matrix of the composed map
\[
\mathcal O^N_{B^+}\xrightarrow{q_\infty}\mathcal F_\infty\hookrightarrow\mathcal E_\infty.
\]
Since $Q$ is bounded on $\overline B$, multiplication by $Q$ defines a bounded linear operator $T_\infty:\mathcal D\to V$. Then we construct a bounded right inverse:
\begin{proposition}[Existence of bounded lifting]
    There exists a bounded linear operator $B_\infty: V\to\mathcal D$, such that
    \[
    T_\infty B_\infty=\mathrm{id}_V.
    \]
\end{proposition}
\begin{proof}
    Since $V$ is closed, by bounded inverse theorem it suffices to show $T_\infty$ is surjective.

Let $f\in V$ be an arbitrary element. By Cartan's theorem B, there exists $b\in H^0(B,\mathcal O^N)$, such that $Qb=f$. We then correct the potentially bad boundary behavior of $b$, so that it is $L^2$-integrable.

Since away from $0$ $\mathcal F_\infty$ is locally-free, we know that $Q$ has full rank away from $0$. Thus we consider the fiberwise linear map
\[
Q^{\dagger}=Q^*(QQ^*)^{-1},
\]
which is defined away from $0$ and has bounded derivative on any pre-compact annulus. Now let $\eta$ be a cut-off function supported in $B$ and equals to $1$ near $0$, so we can define
\[
a_0=\eta b+(1-\eta)Q^{\dagger}f.
\]
It is clear that
\[
Qa_0=f,\quad a_0,\pp a_0\in L^2(B),\quad \pp a_0=0\ \text{near}\ 0.
\]
Since $Q,f$ are holomorphic, it also follows that
\begin{equation}\label{equa--dbar a_0 belongs to ker Q}
    Q\pp a_0=0.
\end{equation}

Denote by $R=R_\infty$ the holomorphic matrix of inclusion $\iota:\ker Q\hookrightarrow\mathcal O^N_{B^+}$. Then $R$ also has full rank away from $0$, we can define
\[
\beta=(R^*R)^{-1}R^*\pp a_0\in L^2(B).
\]
Here we use the fact that $\pp a_0$ vanishes near $0$. The map $R(R^*R)^{-1}R^*$ is the orthogonal projection to $\Im R$, so by \eqref{equa--dbar a_0 belongs to ker Q} we know
\[
\pp a_0=R\beta.
\]
Differentiate this we get $R\pp\beta=0$ away from $0$, thus $\pp\beta=0$ on $B^{\setminus}\{0\}$. By definition of $\beta$, we conclude that $\pp\beta=0$ on $B$.

Since $\dim_{\mathbb C}X=2$, by standard H\"ormander's estimate we can find $w\in L^2(B,\mathbb C^k)$, such that
\[
\pp w=\beta.
\]
We then define
\[
a=a_0-Rw,
\]
which satisfies
\[
\pp a=0,\quad Qa=f,\quad a\in L^2(B).
\]
Thus $a\in\mathcal D$ and $T_\infty (a)=f$.
\end{proof}

Next we transport all spaces of $L^2$-holomorphic sections into one common space. Let $u_j: E_\infty\vert_{U\setminus\{0\}}\to E\vert_{U\setminus\{0\}}$ be the unitary gauge realizing the Uhlenbeck convergence, i.e. $\{u_j^* A_j\}$ converges to $A_\infty$ in $C^{\infty}_{loc}(U\setminus\{0\})$. Define
\[
\mathcal H\coloneqq L^2(B, E_\infty),\quad H_j\coloneqq u_j^* A^2(B,\mathcal E_j)\subset\mathcal H.
\]
Note that each $H_j$ is closed in $\mathcal H$, since $u_j^*$ is an isometry and $A^2(B,\mathcal E_j)$ is complete. We can also regard $H_\infty$ as a closed subspace in $\mathcal H$.

Define
\[
T_j:\mathcal D\to H_j,\quad T_ja=u_j^*(q_j a).
\]
By the estimate \eqref{equa--uniform interior estimate for q_j}, the operator norm of $T_j$ is uniformly bounded. Regarding them as an operator on the common space $\mathcal H$, we have:
\begin{lemma}\label{lem--operator convergence}
    As $j\to\infty$, $\Vert T_j-T_\infty\Vert_{\mathrm{op}}\to 0$.
\end{lemma}
\begin{proof}
    Fix a small ball $B_\rho\subset B$. Then by Cauchy's estimate, for any $a\in\mathcal D$ we have
    \begin{equation}\label{equa--interior estimate in convergence estimate for T_j}
        \sup_{B_\rho}\vert a\vert\leq C\Vert a\Vert_{L^2(B)}.
    \end{equation}

 Let $0<\delta<\rho$ be determined later. By estimate \eqref{equa--uniform interior estimate for q_j} and \eqref{equa--interior estimate in convergence estimate for T_j}, we have
    \begin{equation}\label{equa--norm estimate on B_delta}
        \Vert(T_j-T_\infty)a\Vert_{L^2(B_\delta)}\leq C\delta^2\Vert a\Vert_{L^2(B)}.
    \end{equation}
    On the other hand, columns of $q_j$ converge to $q_\infty$ uniformly on $\overline{B}\setminus B_{\delta}$, so
    \begin{equation}\label{equa--norm estimate outside B_delta}
        \Vert (T_j-T_\infty)a\Vert_{L^2(B\setminus B_\delta)}=\Psi(j^{-1}\vert\delta)\Vert a\Vert_{L^2(B)}.
    \end{equation}
    Thus for any $\epsilon>0$, we choose $\delta$ such that $C\delta^2<\epsilon$, and then choose $j\geq j_0$ such that $\Psi(j^{-1}\vert\delta)<\epsilon$. It follows from \eqref{equa--norm estimate on B_delta} and \eqref{equa--norm estimate outside B_delta} that $\Vert T_j-T_\infty\Vert_{\mathrm{op}}<2\epsilon$ for $j\geq j_0$, thus finishing the proof.
\end{proof}

Let $P_0:\mathcal H\to\mathcal H$ be the orthogonal projection to $H_\infty$, then we have
\begin{lemma}\label{lem--convergence of Hj to Hinfty}
    As $j\to\infty$, $\Vert(I-P_0)\vert_{H_j}\Vert_{\mathrm{op}}\to 0$.
\end{lemma}
\begin{proof}
    Let $\hat s=u_j^* s\in H_j$ be an element with $\Vert\hat s\Vert\leq 1$, where $s$ is a holomorphic section of $\mathcal E_j$. Then by Moser iteration, we fix one $\rho$, so for any $\delta<\rho$ sufficiently small to be determined later, we have a uniform estimate
    \begin{equation}\label{equa--interior estimate used in convergece of Hj to Hinfty}
        \sup_{B_{2\delta}}\vert \hat s\vert_h\leq C=C(\rho).
    \end{equation}

    Choose a cut-off $\chi_\delta$ equal to zero on a neighborhood of $\overline{B_\delta}$ and equal to one outside $B_{2\delta}$. Moreover, assume $\vert d\chi_{\delta}\vert\leq\frac{C}{\delta}$. Put $v=\chi_{\delta}\hat s$. Then by \eqref{equa--interior estimate used in convergece of Hj to Hinfty} we know
    \begin{equation}\label{equa--closeness of hat s and v}
        \Vert\hat s-v\Vert_{L^2(B)}\leq C\delta^2.
    \end{equation}

    Let $b_j=u_j^*\pp_j-\pp_\infty$, then 
    \[
    \Vert b_j\Vert_{L^{\infty}(\overline B\setminus B_\delta)}=\Psi(j^{-1}\vert\delta).
    \]
    Since $u_j^*\pp_j\hat s=0$ on $B\setminus B_{\delta}$, we can estimate
    \[
    \begin{aligned}
            \Vert\pp_\infty v\Vert_{L^2}&\leq\Vert\pp\chi_\delta\wedge\hat s\Vert_{L^2}+\Vert\chi_\delta b_j\hat s\Vert_{L^2}\\
            &\leq C\delta+\Psi(j^{-1}\vert\delta).
        \end{aligned}
    \]
    Now we apply H\"ormander's estimate to the \textit{fixed} holomorphic vector bundle $\mathcal E_\infty$, so we can find $u\in L^2(B,E_\infty)$, such that
    \begin{equation}\label{equa--perturb v by hormander}
         \pp_\infty u=\pp_\infty v,\quad \Vert u\Vert_{L^2}\leq C\Vert\pp_\infty v\Vert_{L^2}\leq C\delta+\Psi(j^{-1}\vert\delta).
    \end{equation}
    It follows that $t=v-u\in H_\infty$. Moreover, by \eqref{equa--closeness of hat s and v} and \eqref{equa--perturb v by hormander} we know
    \[
    \Vert (I-P_0)\hat s\Vert=d(\hat s,H_\infty)\leq C\delta+\Psi(j^{-1}\vert\delta).
    \]
    Since $\hat s$ is arbitrary, we conclude that
    \[
    \Vert(I-P_0)\vert_{H_j}\Vert\leq C\delta+\Psi(j^{-1}\vert\delta).
    \]
    We can then first choose $\delta$ small, then choose $j$ large to conclude that $\lim_{j\to\infty}\Vert(I-P_0)\vert_{H_j}\Vert=0$.
\end{proof}

The following is standard:
\begin{lemma}\label{lem--weak compactness for bounded section in Hj}
    Every bounded sequence $s_j\in H_j$ has a weakly convergent subsequence in $\mathcal H$, whose limit is contained in $V$.
\end{lemma}
\begin{proof}
    The weak convergence follows from Moser iteration and elliptic estimate; see discussion in the beginning of Section \ref{sec-L2 estimate for HYM connections}. The last statement follows from Lemma \ref{lem--saturation of algebraic limit}.
\end{proof}

Let $P:\mathcal H\to\mathcal H$ be the orthogonal projection to $V\subset H_\infty$. Then we can improve Lemma \ref{lem--convergence of Hj to Hinfty} in the following way:
\begin{proposition}\label{prop--upper gap tends to zero}
    As $j\to\infty$, $\Vert (I-P)_{H_j}\Vert_{\mathrm{op}}\to 0$.
\end{proposition}
\begin{proof}
    Let $W$ be the orthogonal complement of $V$ in $H_\infty$ and $P_W:\mathcal H\to\mathcal H$ be the orthogonal projection to $W$. We claim that $\Vert P_W\vert_{H_j}\Vert_{\mathrm{op}}\to 0$ as $j\to\infty$.

    Suppose not. Then we can find a subsequence of sections $s_j\in H_j$ such that
    \[
    \Vert P_Ws_j\Vert_{L^2}\geq c>0,\quad \Vert s_j\Vert_{L^2}=1.
    \]
    By Lemma \ref{lem--weak compactness for bounded section in Hj} we may assume $s_j$ weakly converges to a limit $s_\infty\in V$. By Corollary \ref{cor--L2-convergence for sections with high vanishing order}, we know that $W$ is finite dimensional. Therefore $P_Ws_j$ strongly converges to $P_Ws_\infty=0$, contradiction.

    Now by orthogonal decomposition
    \[
    \mathcal H= V\oplus W\oplus H_\infty^{\perp}
    \]
    implies that
    \[
    \Vert(I-P)s\Vert^2=\Vert(I-P_0)s\Vert^2+\Vert P_W\Vert^2.
    \]
    The result then follows from Lemma \ref{lem--convergence of Hj to Hinfty} and our claim.
\end{proof}

\begin{corollary}\label{cor--injectivity of P restricted to Hj for j large}
    For $j$ sufficiently large, $P\vert_{H_j}$ is injective.
\end{corollary}
\begin{proof}
    Suppose $0\neq h\in H_j$ with $Ph=0$. Then by Proposition \ref{prop--upper gap tends to zero}
    \[
    0<\Vert h\Vert=\Vert(I-P)h\Vert\leq\Vert(I-P)\vert_{H_j}\Vert_{\mathrm{op}}\cdot\Vert h\Vert<\Vert h\Vert,
    \]
    contradiction.
\end{proof}

We are now ready to prove Theorem \ref{thm--local convergence in a flat family}.
\begin{proof}[Proof of Theorem \ref{thm--local convergence in a flat family}]

   Define a sequence of operators on the fixed space $V$ by
    \[
    C_j\coloneqq PT_j B_\infty:V\to V.
    \]
    Since $PT_\infty B_\infty=\mathrm{id}_V$, Lemma \ref{lem--operator convergence} implies that
    \[
    \Vert C_j-I\Vert_{\mathrm{op}}\leq\Vert B_\infty\Vert_{\mathrm{op}}\cdot\Vert T_j-T_\infty\Vert_{\mathrm{op}}\to 0
    \]
    as $j\to\infty$. Thus $C_j$ is invertible for $j$ large with $\Vert C_j^{-1}\Vert_{\mathrm{op}}\leq 2$.

    Now define the bounded lifting operator:
    \[
    B_j\coloneqq B_\infty C_j^{-1}P\vert_{H_j}: H_j\to\mathcal D,
    \]
    whose operator norm is uniformly bounded by $2\Vert B_\infty\Vert_{\mathrm{op}}$. We claim that it is a right inverse of $T_j$ on $H_j$. To see this, for any $h\in H_j$ we have
    \[
    P(T_j B_jh-h)=C_jC_j^{-1}Ph-Ph=0.
    \]
    By Corollary \ref{cor--injectivity of P restricted to Hj for j large}, we must have $T_jB_jh=h$. In particular, $T_j$ is surjective for all $j$ large. Moreover, for any $s\in A^2(B,\mathcal E_j)$, there exists $a\in\mathcal D$ such that
    \[
    q_ja=s,\quad \Vert a\Vert_{L^2(B)}\leq C\Vert s\Vert_{L^2(B)}.
    \]

    We finally return to sheaf theoretical perspective. We first show that $q_j$ is surjective as a sheaf morphism. For any $z\in B$, by Cartan's theorem A on $B^+$, there exist holomorphic sections $s_i\in H^0(B^+,\mathcal E_j)$, such that the stalk at $z$ is generated by $\{s_{i,z}\}$. Each section is $L^2$ over $B$, so we can find $a_i\in\mathcal D$, such that $q_ja_i=s_i$. Thus columns of $q_j$ span the stalk at $z$.

    Next we show the relation matrix also converges on a possibly smaller ball. Let $\rho_\infty$ be any column of $R_\infty$, which is a $\mathbb C^N$-valued holomorphic function over $B^+$, thus can be regarded as an element in $\mathcal D$. Moreover, $T_\infty\rho_\infty=0$. Now define
    \[
    \rho_j=\rho_\infty-B_jT_j\rho_\infty\in\mathcal D,
    \]
    which satisfies
    \[
    T_j\rho_j=T_j\rho_\infty-T_jB_jT_j\rho_\infty=T_j\rho_\infty-T_j\rho_\infty=0.
    \]
    It follows that $q_j\rho_j=0$ as holomorphic sections on $B$. Furthermore, by Lemma \ref{lem--operator convergence}
    \[
    \Vert\rho_j-\rho_\infty\Vert_{L^2(B)}\leq 2\Vert B_\infty\Vert_{\mathrm{op}}\cdot\Vert T_j-T_\infty\Vert_{\mathrm{op}}\Vert\rho_\infty\Vert\to 0
    \]
    as $j\to\infty$.

    We then apply this construction to \textit{every} column of $R_\infty$, so we get a holomorphic matrix $R_j$ with $q_jR_j=0$. In this way we get a morphism
    \[
    \mathcal O^k_{B}\xrightarrow{R_j}\mathcal O^N_{B},
    \]
    whose image is contained in $\ker q_j$.
    
    The construction here also shows that $\Vert R_j-R_\infty\Vert_{L^2(B)}\to 0$, so by shrinking the ball $B^{\prime}\subset\subset B$ and passing to a subsequence, we can assume $R_j\to R_\infty$ in $C^{\infty}_{loc}(B^{\prime})$. It remains to show $R_j$ induces an isomorphism 
    \[
    R_j:\mathcal O^k_{B^{\prime}}\xrightarrow{\sim}\ker q_j\vert_{B^{\prime}}
    \]
    for all $j$ large.

    Since $q_j$ is surjective, we know that $\ker q_j$ is locally-free, so we may trivialize to write $\det R_j$ as a single holomorphic function. It follows that the zero locus of determinant section of $R_j$ cannot be isolated points. For any $\rho$ small, by smooth convergence we know that $R_j$ is an isomorphism over $B^{\prime}\setminus B_\rho$, so the determinant section is not identically zero. Thus the zero locus can only be either empty or a curve. The latter case is also impossible, since $B^{\prime}$ is Stein and cannot contain any compact positive dimension subvariety. This finishes the proof of Theorem \ref{thm--local convergence in a flat family} except for the identification between algebraic multiplicity and analytic multiplicity, which will be proved below.
\end{proof}

\subsubsection{Computing algebraic multiplicity}\label{subsec--computing algebraic multiplicity}

Now we compute the algebraic multiplicity of $\mathcal F_\infty$, thus finishing the proof of Theorem \ref{thm--local convergence in a flat family}. Since the singularity is isolated on the surface base, we can simply use the relative $K$-theory on a closed ball for the calculation of topological invariants.

Recall that for a pair of topological spaces $(X,Y)$, there is a well-defined relative $K$-theory \cite{AH61}. In fact, it is defined through relative $K$-theory, i.e.
\[
K(X,Y)\coloneqq\widetilde K(X/Y).
\]
Here $Y$ is a closed subspace of a compact topological space $X$. The relative Chern character is thus also well-defined via the relative cohomology theory, i.e.
\[
\ch: K^0(X,Y)\to H^{\mathrm{even}}(X,Y;\mathbb Q).
\]

By \cite[Theorem 2.6.1]{Atiyah67}, element in $K^0(X,Y)$ can be represented by a triple $(E^+, E^-,\sigma)$, where
\[
\sigma:E^+\vert_Y\xrightarrow{\sim} E^-\vert_Y
\]
is an isomorphism between complex vector bundles. Specializing to our setting, we fix a smooth disk $D$ centered at the bubbling point $0$. Then the limit presentation
\[
    0\to\mathcal O_{B^+}^k\xrightarrow{R_\infty}\mathcal O_{B^+}^N\xrightarrow{q_\infty}\mathcal F_\infty\to 0.
    \]
    defines a triple $(\underline{\mathbb C}^k\oplus E_\infty,\underline{\mathbb C}^N,\sigma_\infty^{-1})$, where
    \[
    \sigma_\infty=R_\infty^*+q_\infty:\left(\underline{\mathbb C}^N\vert_{\partial D}\to \underline{\mathbb C}^k\oplus E_\infty\vert_{\partial D}\right)
    \]
    is an isomorphism between complex vector bundles. Here $E_\infty$ is the underlying complex bundle of $\mathcal E_\infty$. Denote by $\mathcal Q=\mathcal F_{\infty}^{\vee\vee}/\mathcal F_{\infty}$, so this construction defines a relative class $\eta_{\mathcal  Q}\in K^0(D,\partial D)$. The following is the local analogy of Proposition \ref{prop--chern character of torsion sheaf}.
    \begin{lemma}\label{lem--length of quotient as relative chern character}
        Under the natural complex orientation, we have
        \[
        \langle\ch_2(\eta_{\mathcal Q}),[D,\partial D]\rangle=\mathrm{length}_{\mathcal O_{D,0}}(\mathcal Q_0).
        \]
    \end{lemma}
    \begin{proof}
        The proof is standard, so we will be sketchy. Denote by $\ell=\mathrm{length}_{\mathcal O_{D,0}}(\mathcal Q)$. By choosing a composition series of $\mathcal Q$ we can conclude that
        \[
       \eta_{\mathcal Q}=\ell \eta_{\mathbb C_0},
        \]
        where $\mathbb C_0$ is the skyscraper sheaf. The latter admits the Koszul resolution
        
        \[0\longrightarrow\mathcal O
\xrightarrow{\binom{-z_2}{z_1}}
\mathcal O^{\oplus2}
\xrightarrow{(z_1\;\;z_2)}
\mathcal O
\longrightarrow\mathbb C_0
\longrightarrow0,\]
from which we can compute explicitly that $ 
        \langle\ch_2(\eta_{\mathbb C_0}),[D,\partial D]\rangle=1.
        $
        \end{proof}

   Now by the discussion in Section \ref{subsubsec--Hilbert space method}, we know 
   \[
   \sigma_j=R_j^*+q_j:\underline{\mathbb C}^N\to\underline{\mathbb C}^k\oplus E_j
   \]
   is an isomorphism on the \textit{whole} $D$ for all $j$ large, so
   \[
   [\underline{\mathbb C}^N,\underline{\mathbb C}^k\oplus E_j,\sigma_j]=0\in K^0(D,\partial D)
   \]
   for all $j$ large. On the other hand, since $q_j$ converges to $q_\infty$ in the Uhlenbeck sense, we have
   \[
   (\mathrm{id}_{\underline{\mathbb C}^k}\oplus u_j^{-1})\sigma_j\to\sigma_\infty\ \text{in}\ C^{\infty}\ \text{near a neighborhood of}\ \partial D.
   \]
   Thus the boundary isomorphisms $\sigma_j\sigma_\infty^{-1},\mathrm{id}_{\underline{\mathbb C}^k}\oplus u_j$ are homotopic for all $j$ large. Combined this discussion we get for all $j$ large
   \[
   \begin{aligned}
        \eta_{\mathcal Q}&=[\underline{\mathbb C}^k\oplus E_\infty,\underline{\mathbb C}^N,\sigma_\infty^{-1}]\\
        &=[\underline{\mathbb C}^k\oplus E_\infty,\underline{\mathbb C}^N,\sigma_\infty^{-1}]+[\underline{\mathbb C}^N,\underline{\mathbb C}^k\oplus E_j,\sigma_j]\\
        &=[\underline{\mathbb C}^k\oplus E_\infty,\underline{\mathbb C}^k\oplus E_j,\sigma_j\sigma_\infty^{-1}]\\
        &=[\underline{\mathbb C}^k\oplus E_\infty,\underline{\mathbb C}^k\oplus E_j,\mathrm{id}_{\underline{\mathbb C}^k}\oplus u_j]\\
        &=[E_\infty, E_j,u_j].
   \end{aligned}
  \]
  The result then follows from Lemma \ref{lem--length of quotient as relative chern character} and the standard Chern-Simons transgression argument.

\subsection{Global construction}\label{subsec--global construction of flat families}

In this subsection, we globalize the local flat deformation to prove Theorem \ref{thm-global convergence in a flat family}.

Recall that a coherent analytic sheaf $\mathcal G$ on a complex space $X$ can be described by a set of \textit{defining sheaf morphisms} and a set of \textit{relating sheaf morphisms} \cite[Introduction]{siu-trautmann81}. To be more precise, let $\{V_\alpha\}$ be a locally finite open cover of $X$. We say $R_\alpha:\mathcal O^{k_\alpha}\to\mathcal O^{N_\alpha}$ on $V_\alpha$ is a defining sheaf morphism for $\mathcal G$ on $V_\alpha$ if there exists a morphism
\[
q_\alpha:\mathcal O^{N_\alpha}\to\mathcal G\vert_{V_\alpha}
\]
which induces an isomorphism between $\mathrm{coker} R_\alpha$ and $\mathcal G\vert_{V_\alpha}$.

On the other hand, we say that $A_{\alpha\beta}:\mathcal O^{N_\beta}\to\mathcal O^{N_\alpha}$ on $V_{\alpha\beta}\coloneqq V_\alpha\cap V_\beta$ is a relating sheaf morphism for $R_\alpha$ and $R_\beta$, if $A_{\alpha\beta}$ induces an isomorphism between $\mathrm{coker} R_\beta$ and $\mathrm{coker} R_\alpha$ on $V_{\alpha\beta}$. In particular, $q_\alpha A_{\alpha\beta}=q_\beta$ on $V_{\alpha\beta}$.

Note that in contrast to the bundle case, the transition data $\{A_{\alpha\beta}\}$ is not necessarily a cocycle. To see this, on a triple overlap $V_{\alpha\beta\gamma}$ we have
\[
q_\alpha A_{\alpha\beta}A_{\beta\gamma}=q_\beta A_{\beta\gamma}=q_\gamma=q_\alpha A_{\alpha\gamma},
\]
so their difference takes value in $\ker q_\alpha=\Im R_\alpha$. Suppose now $\{V_\alpha\}$ is a Stein cover, so any finite intersection is again Stein \cite[Chapter V, §1.2]{grauert-remmert79}. Therefore by Cartan's theorem B, we can choose a lift on $V_{\alpha\beta\gamma}$
\[
H_{\alpha\beta\gamma}:\mathcal O^{N_\gamma}\to\mathcal O^{k_\alpha},
\]
such that
\[
A_{\alpha\beta}A_{\beta\gamma}-A_{\alpha\gamma}=R_\alpha H_{\alpha\beta\gamma}.
\]
Similarly, we can also choose some $B_{\alpha\beta}:\mathcal O^{k_\beta}\to\mathcal O^{k_\alpha}$ on $V_{\alpha\beta}$, such that
\[
A_{\alpha\beta}R_\beta=R_\alpha B_{\alpha\beta}.
\]

Note that both $\{B_{\alpha\beta}\}$ and $\{H_{\alpha\beta\gamma}\}$ may not be unique in general. While this seems to be an interesting question to address in a more general setting, we will use our surface geometry later to avoid further complications. Since our base is a smooth surface and the target sheaf $\mathcal G$ is torsion-free, it follows that we can assume each $R_\alpha$ is an injective sheaf morphism, thus $\{B_{\alpha\beta}\}$ and $\{H_{\alpha\beta\gamma}\}$ are uniquely determined by $\{q_\alpha\}$ and $\{A_{\alpha\beta}\}$. This fact will be used later.

The precise local-to-global passage we need is the following: 
\begin{theorem}\label{thm--local deformation to global deformation}
    Let $X$ be a compact complex manifold and $\mathcal G_0\to X$ be a coherent sheaf with local projection dimension of $\mathcal G_0$ at most $1$. Let $\{V_\alpha\}$ be a finite Stein open cover of $X$, over which we have length $1$ locally-free resolutions of $\mathcal G_0$, i.e. a sheaf exact sequence on $U_\alpha$
    \[
    0\to\mathcal O^{k_\alpha}\xrightarrow{R_{\alpha,0}}\mathcal O^{N_\alpha}\xrightarrow{q_{\alpha,0}}\mathcal G_0\vert_{U_\alpha}\to 0.
    \]
    We also fix a set of relating morphisms $\{A_{\alpha\beta,0}\}$ on $V_{\alpha\beta}$.

    Let $\{\mathcal G_j\}_{j\geq 1}$ be a set of coherent sheaves of local projective dimension at most $1$, and $\{q_{\alpha,j}\},\{A_{\alpha\beta,j}\}$ be a given set of defining data of $\mathcal G_j$. Suppose 
    \[
\begin{aligned}
R_{\alpha,j}&\longrightarrow R_{\alpha,0}
&&\text{in }C^\infty_{\mathrm{loc}}(V_\alpha),\\
 A_{\alpha\beta,j}
&\longrightarrow A_{\alpha\beta,0}
&&\text{in }C^\infty_{\mathrm{loc}}(V_{\alpha\beta}).
\end{aligned}
\]
Then there exists a finite dimensional complex analytic space germ $(S,0)$, a coherent sheaf $\mathcal G\to X\times (S,0)$ flat over $S$ and a sequence of points $S\ni s_j\to 0$, such that
\[
\mathcal G_0\cong\mathcal G\vert_{X\times\{0\}},\quad \mathcal G_j\cong\mathcal G\vert_{X\times\{s_j\}}, \forall j\gg 1.
\]
\end{theorem}

The proof of Theorem \ref{thm--local deformation to global deformation} is based on the technique of Banach analytic spaces, as developed in the seminal work of Douady \cite{douady66,douady74}. As the proof is lengthy, we postpone it to the Appendix \ref{sec--proof of local-to-global passage} and hereby finish the proof of Theorem \ref{thm-global convergence in a flat family} by assuming Theorem \ref{thm--local deformation to global deformation}.

\begin{proof}[Proof of Theorem \ref{thm-global convergence in a flat family}]
    We first choose a sufficiently small finite Stein cover $V_\alpha$, such that each bubbling point belongs to exactly one of the $V_\alpha$ and the conclusion of Theorem \ref{thm--local convergence in a flat family} holds on such $V_\alpha$. In particular, we can assume the Uhlenbeck convergence is smooth on any overlap $\overline{V_{\alpha\beta}}$ and $\mathcal F_\infty\vert_{V_{\alpha\beta}}=\mathcal E_\infty\vert_{V_{\alpha\beta}}$. For any other $V_\alpha$ where the Uhlenbeck convergence is smooth, by usual H\"ormander's estimate, we can simply take $k_\alpha=0, N_\alpha=\rk\mathcal E$ and $q_{\alpha,j}$ is given by a sequence of convergent holomorphic bases. We may need to refine the cover, but we stick to the notation $\{V_\alpha\}$. According to Theorem \ref{thm--local deformation to global deformation}, it suffices to construct convergent relating morphism $A_{\alpha\beta,j}$ on $V_{\alpha\beta}$.

    Since $V_{\alpha\beta}$ is Stein, by Cartan's theorem B we can choose a holomorphic splitting
    \[
    P_{\alpha\beta}:\mathcal E_\infty\vert_{V_{\alpha\beta}}\to\mathcal O^{N_\alpha},\quad q_{\alpha,\infty}P_{\alpha\beta}=\mathrm{Id}_{\mathcal E_\infty}.
    \]
    We then define
    \[
    T_{\alpha\beta,j}\coloneqq q_{\alpha,j}P_{\alpha\beta}:\mathcal E_\infty\vert_{V_{\alpha\beta}}\to\mathcal E_j\vert_{V_{\alpha\beta}}.
    \]
    Since $q_j$ converges to $q_\infty$ in the Uhlenbeck sense, by possibly refine the cover, we have
    \[
    u_j^* T_{\alpha\beta,j}\to\mathrm{Id}\quad\text{in}\ C^{\infty}(\overline{V_{\alpha\beta}}).
    \]
    It follows that for $j$ large, $T_{\alpha\beta,j}$ is an isomorphism. Thus we can define
    \[
   \begin{aligned}
       &A_{\alpha\beta,j}\coloneqq P_{\alpha\beta}T_{\alpha\beta,j}^{-1}q_{\beta,j}:\mathcal O^{N_\beta}\to\mathcal O^{N_\alpha}&& \text{on}\ V_{\alpha\beta},\\
       & A_{\alpha\beta, \infty}\coloneqq P_{\alpha\beta}q_{\beta,\infty}:\mathcal O^{N_\beta}\to\mathcal O^{N_\alpha}&&\text{on}\ V_{\alpha\beta}.
   \end{aligned}
    \]
    It follows that
    \[
    q_{\alpha,j}A_{\alpha\beta,j}=q_{\beta,j},\quad A_{\alpha\beta,j}\to A_{\alpha\beta,\infty}\ \text{in}\ C^{\infty}_{loc}(V_{\alpha\beta}).
    \]
    This finishes the proof.
\end{proof}

\section{Geometric consequences}

In this section, we prove Theorem \ref{thm--continuity theorem} and Theorem \ref{thm-boundedness of semistable sheaves with fixed topology}.

\subsection{A reduction lemma}

We will use the following reduction lemma, which is of independent interest. The proof is partially inspired by \cite[Proposition 3.6]{BTT2017}.

\begin{lemma}[Common disk reduction]\label{lem--reduction to common disk}
    Let $X$ be a compact complex manifold, and $\mathcal A,\mathcal B$ be two flat families of torsion-free sheaves on $X\times(S,s_0)$ and $X\times (T,t_0)$, respectively. Suppose there exist sequences
    \[
    S\setminus\{s_0\}\ni s_j\to s_0,\quad T\setminus\{t_0\}\ni t_j\to t_0,\quad \mathcal A_{s_j}\cong\mathcal B_{t_j}.
    \]
    Then we can find non-constant holomorphic maps of germs
    \[
    \mathfrak a:(\Delta,0)\to(S,s_0),\quad \mathfrak b:(\Delta,0)\to(T,t_0),
    \]
    such that 
    \[
    \mathfrak a^*\mathcal A\vert_{X\times\Delta^*}\cong \mathfrak b^*\mathcal B\vert_{X\times\Delta^*}.
    \]
    In particular, if $(S,s_0)=(\Delta,0)$, $\mathfrak a$ is a finite map of disk germs.
\end{lemma}

Note that in general we \textit{cannot} arrange the selected disk containing even some subsequence of the original sequence. The main point is to produce two isomorphic punctured families with prescribed central fibers.

\begin{proof}[Proof of Lemma \ref{lem--reduction to common disk}]
    Put $W=S\times T$ and pull back both families to $X\times W$, denoted by $\mathcal A_W$ and $\mathcal B_W$. Consider the following functor from the category of complex analytic spaces over $W$ to the category of sets
    \[
    \begin{aligned}
        H\coloneqq&\underline{\mathrm{Hom}}(\mathcal A_W,\mathcal B_W):\mathfrak{An}_W\to\mathfrak{sets},\\
    &(V\xrightarrow{f}W)\mapsto\mathrm{Hom}_{X\times V}(f^*\mathcal A_W,f^*\mathcal B_W).
    \end{aligned}
    \]
   Since $X$ is compact and $\mathcal B_W$ is flat over $W$, by \cite[Section 3.2]{flenner81a} and \cite{bingener80} the functor $H$ is representable by a linear fiber space $\pi: H\to W$. That is to say, $H$ is a complex analytic space whose fibers are linear spaces, while allowing possibly jump of fiberwise dimensions. Locally on the base $W$, we can write 
   \[
   H=\{(w,z)\in W\times\mathbb C^N: D(w)z=0\},
   \]
   where $D(w)$ is a matrix of holomorphic functions. The intrinsic scalar multiplication on $H$ defines a projectivization $P=\mathbb P(H)$, so the natural projection $p:P\to W$ is proper. 
   
   On the other hand, the representability means that for any holomorphic map $f: V\to W$, there is a natural bijection between
   \[
   \{g: V\to H\ \text{with}\ \pi\circ g=f\}\cong\mathrm{Hom}_{X\times V}(\mathcal A_V,\mathcal B_V).
   \]
   Here $\mathcal A_V,\mathcal B_V$ are natural pull-back via $f$. That is, choosing a holomorphic lift to $H$ is equivalent to choose a morphism of the pull-back families. In particular, take $V=\{(s,t)\}$ with the natural inclusion $\iota:V\to W$, we conclude that
   \[
   H_{(s,t)}=\mathrm{Hom}_X(\mathcal A_s,\mathcal B_t).
   \]

   Now by our assumption, we can find isomorphisms $\varphi_j:\mathcal A_{s_j}\to\mathcal B_{t_j}$. The set $\{(s_j,t_j)\}_j$ is compact by convergence, so properness of $p:P\to W$ implies that we can take limit for
   \[
   p_j\coloneqq[\varphi_j]\to p_0\in P_{(s_0,t_0)}.
   \]
It is not hard to see in a small neighborhood $V$ of $p_0$, we can choose a holomorphic lift $V\to H\setminus(0\text{-section})$ \cite[Lemma 3.7]{BTT2017}. Thus by representability, this corresponds to an actual morphism
\[
\Psi:\mathcal A_V\to\mathcal B_V\ \text{on}\ X\times V.
\]

Let $q:X\times V\to V$ be the projection, and set 
\[
D=q\left(\mathrm{supp}(\mathrm{coker}\Psi)\right).
\]
Since $q$ is proper, by Remmert's proper mapping theorem we know $D$ is a closed analytic subset. We claim that $V\setminus D$ is exactly the fiberwise isomorphic locus.

To see this, note that restriction to fiber is right exact, so for any $v\in V$,
\[
\mathrm{coker}(\Psi_v)\cong\mathrm{coker}\Psi\vert_{X\times\{v\}}.
\]
Moreover, by Nakayama's lemma $\big(\mathrm{coker}\Psi\vert_{X\times\{v\}}\big)_x=0$ if and only if $(x,v)\notin\mathrm{supp}(\mathrm{coker}\Psi)$. Therefore $v\notin D$ is equivalent to $\Psi_v$ is surjective on $X$. Since $\mathcal A_v,\mathcal B_v$ has the same rank and $\mathcal A_v$ is torsion-free, any surjection between them must be an isomorphism. This proves the claim.

Now we enlarge $D$ by considering
\[
D^{\prime}=D\cup(p\vert_V)^{-1}(\{s_0\}\times T))\cup(p\vert_V)^{-1}(S\times\{t_0\}),
\]
which is still a closed analytic subset. By our assumption, $p_j\in V\setminus D^{\prime}$ for all $j$ large, so $p_0\in\overline{V\setminus D^{\prime}}$. On the other hand, our choice of added components in $D^{\prime}$ guarantees that $p_0\in D^{\prime}$. Then after shrinking $V$, one irreducible component $Y$ must contain a subsequence, so in particular $Y\not\subset D^{\prime}$. It follows that
\[
\dim(Y\setminus D^{\prime})=\dim Y\geq 1.
\]
Thus we may choose a curve germ in $Y$ intersecting $D^{\prime}$ exactly at $p_0$. By taking its normalization, we obtain a non-constant holomorphic map
\[
\gamma:(\Delta,0)\to(Y,p_0),\quad \gamma(\Delta^*)\subset V\setminus D^{\prime}.
\]
The projections of $p\circ\gamma$ to $S$ and $T$ give the desired map $\mathfrak a,\mathfrak b$ with $\mathfrak a(0)=s_0,\mathfrak b(0)=t_0$. In particular, it cannot be a constant map. Pulling back $\Psi$ gives fiberwise isomorphism over $\Delta^*$, which then gives an isomorphism between families by flatness and Nakayama's lemma.

The last sentence follows from the fact that any non-constant holomorphic map between disk germs must be a finite map.
\end{proof}

\subsection{Proof of Theorem \ref{thm--continuity theorem}}

In this subsection, we prove Theorem \ref{thm--continuity theorem}.

We start with the following lemma, which was proved in \cite[Proposition 2.15]{GSTW2021}:
\begin{lemma}\label{lem--gstw prop 2.15}
    Suppose $\mathcal F$ is torsion-free such that $\mathcal F^{\vee\vee}$ is $\mu$-polystable. Then $\mathcal E$ is $\mu$-semistable. Moreover, we have identifications
    \[
    \gr^S(\mathcal F)^{\vee\vee}\cong\mathcal F^{\vee\vee},\quad \mathcal C(\mathcal F)=\mathcal C(\mathcal F^{\vee\vee}/\mathcal F).
    \]
\end{lemma}

Now we prove the following statement, which is an analogue of \cite[Proposition 4.5 and Proposition 4.7]{GSTW2021}.
\begin{proposition}\label{prop--continuity for smoothable sequences}
    Let $(X,\omega)$ be a compact K\"ahler surface, and $\mathcal G\to X\times (S,s_0)$ be a flat family of torsion-free $\mu$-semistable sheaves. Let $s_j\to s_0$ be a sequence such that $\mathcal G_{s_j}$ is locally-free and $\mu$-stable, and let $A_j$ be its HYM connection on the fixed underlying topological bundle. Then every Uhlenbeck limit $(A_\infty,\mathcal C_\infty)$ satisfies
    \[
    \mathcal E_\infty\cong\gr^S(\mathcal G_{s_0})^{\vee\vee},\quad \mathcal C_\infty=\mathcal C(\mathcal G_{s_0}).
    \]
    Here $\mathcal E_j,\mathcal E_\infty$ denote by holomorphic bundle defined by $A_j$.
\end{proposition}
\begin{proof}
    Without loss of generality, we may assume $s_j\in S\setminus\{s_0\}$ for $j$ large. By passing to some subsequence and relabeling, we assume $A_j$ converges to $(A_\infty,\mathcal C_\infty)$ in the Uhlenbeck sense. Moreover, by Theorem \ref{thm-global convergence in a flat family}, there exists a flat family $\mathcal F\to X\times(T,t_0)$, such that for some $t_j\to t_0$, we have $\mathcal F_{t_j}\cong\mathcal E_j\cong\mathcal G_{s_j}$ and 
    \[
    \mathcal F_\infty^{\vee\vee}\cong\mathcal E_\infty,\quad \mathcal C(\mathcal F_\infty)=\mathcal C_\infty.
    \]
    It follows from Lemma \ref{lem--reduction to common disk}, Theorem \ref{thm--separatedness for slope stability} and Lemma \ref{lem--gstw prop 2.15} that 
    \[
    \gr^S(\mathcal G_{s_0})^{\vee\vee}\cong\gr^S(\mathcal F_\infty)^{\vee\vee}\cong\mathcal F_\infty^{\vee\vee}\cong\mathcal E_\infty,\quad \mathcal C(\mathcal G_{s_0})=\mathcal C(\mathcal F_\infty)=\mathcal C_\infty.
    \]

\end{proof}

\begin{proof}[Proof of Theorem \ref{thm--continuity theorem}]
    As we immediately see from Proposition \ref{prop--continuity for smoothable sequences}, the image of $\overline{S^o}$ under $\Psi_S$ is contained in $\overline{M}_{HYM}$. The proof of continuity also follows from this and a diagonal sequence argument; cf. \cite[Proposition 4.13]{GSTW2021}.
\end{proof}

\subsection{Proof of Theorem \ref{thm-boundedness of semistable sheaves with fixed topology}}

In this subsection, we prove Theorem \ref{thm-boundedness of semistable sheaves with fixed topology}.

The proof is based on an induction argument. First we need the following classical result:

\begin{theorem}\label{thm--uhlenbeck compactification is metrizable and compact}
    $\overline{M}_{HYM}$ is metrizable, hence is compact.
\end{theorem}
\begin{proof}
    This is proved in \cite{donaldson} under the assumption rank $\leq 3$, but is actually true for arbitrary ranks. The key ingredient is the deep result of Taubes on the deformation retraction property of connections on $S^4$ with small $\Vert F^+_A\Vert$. However, since in \cite{donaldson} it only requires the smallness of energy for \textit{some} constant, the proof is insensitive to the choice of gauge groups. See also \cite{waldron16} for a different approach and \cite{DW26} for the effective bound.
\end{proof}
\begin{remark}
    The above theorem does not assert any boundary behavior of the natural $L^2$-Riemannian metric on the smooth locus, which is a much harder problem in geometric analysis; see the survey \cite{groisser1997}. The distance structure constructed in \cite{donaldson} is suitably perturbed to avoid this issue.
\end{remark}

\

On the other hand, we also need a simple fact, which can be loosely stated as follows: sufficiently small perturbation of a bounded set is again bounded. For this we need the following openness of versality of deformations:
\begin{proposition}[Satz (5.4), p. 462 in \cite{flenner81b}]\label{prop-versality is open}
    Let $X,T$ be complex analytic spaces and $\mathcal F$ be a coherent sheaf on $X\times T$ flat over $T$. Assume the support of $\mathcal F$ is proper over $T$. Then 
    \[
    \{t\in T:\mathcal F\ \text{is a versal deformation germ for}\ \mathcal F_t \}
    \]
    is a Zariski open subset of $T$.
\end{proposition}

The point is that the flatness is \textit{not} included in the definition of boundedness, so we will first establish the following flat parametrization lemma. The proof is a standard application of Noetherian induction and generic flatness as pointed out in \cite{toma16}.
\begin{lemma}\label{lem--flat parameterization of bounded collections}
Let $X$ be a compact complex manifold, and let $\mathcal B$ be a bounded
collection of isomorphism classes of coherent sheaves on $X$ in the sense of Definition~\ref{def--boundedness}.
Then there exist finitely many complex analytic spaces $S_i$, compact
subsets $K_i\subset S_i$, and coherent sheaves $\mathcal F_i$ on
$X\times S_i$ flat over $S_i$, such that
\[
    \mathcal B\subset
    \bigcup_i\bigl\{[(\mathcal F_i)_s]:s\in K_i\bigr\}.
\]
\end{lemma}

\begin{proof}
Choose a complex analytic space $S$, a compact subset $K\subset S$,
and a coherent sheaf $\mathcal E$ on $X\times S$ whose fibers over $K$
include $\mathcal B$.

Replacing $S$ by $S_{\mathrm{red}}$ and pulling back $\mathcal E$, we may assume $S$ is reduced. We may also consider irreducible components of $S$ separately, so we assume $S$ is irreducible. We argue by induction on $\dim S$. The zero-dimensional case is immediate.

Since $X$ is compact, the projection $X\times S\to S$ is proper.
By generic flatness and Hironaka's flattening theorem, after shrinking
around $K$ if necessary, there exist a nowhere dense closed analytic
subset $Z\subset S$ and a proper modification
\[
    \pi:\widetilde S\longrightarrow S
\]
which is an isomorphism over $S\setminus Z$, such that the strict
transform $\widetilde{\mathcal E}$ is flat over $\widetilde S$.
Moreover, over $\pi^{-1}(S\setminus Z)$ this strict transform agrees
with $(\operatorname{id}_X\times\pi)^*\mathcal E$.

Properness of $\pi$ implies that
\[
    \widetilde K:=\pi^{-1}(K)
\]
is compact. For every $s\in K\setminus Z$, its unique lift
$\widetilde s\in\widetilde K$ satisfies
\[
    \widetilde{\mathcal E}_{\widetilde s}\simeq\mathcal E_s.
\]
Thus the flat family $\widetilde{\mathcal E}$ over the compact
parameter set $\widetilde K$ accounts for all original fibers over
$K\setminus Z$.

For the remaining fibers, consider the original restricted
family
\[
    \mathcal E_Z:=\mathcal E|_{X\times Z}
\]
over the compact set $K\cap Z$. Its fiber at $z$ is precisely
$\mathcal E_z$. Since $\dim Z<\dim S$, the induction hypothesis
provides finitely many flat families over compact parameter sets
parametrizing these fibers. Together with
$(\widetilde S,\widetilde K,\widetilde{\mathcal E})$, the proof is then complete.
\end{proof}

The following now follows immediately from the existence of versal deformation \cite{siu-trautmann81}, openness of versality in Proposition \ref{prop-versality is open} and Lemma \ref{lem--flat parameterization of bounded collections}:
\begin{corollary}
\label{cor--small perturbation of bounded set is again bounded}
Let $X$ be a compact complex manifold, and let $\mathcal B$ be a
bounded collection of isomorphism classes of coherent sheaves on $X$.
Then we can find a bounded collection
$\widehat{\mathcal B}\supset\mathcal B$ with the following property:
for every pointed complex analytic space $(T,t_0)$ and every coherent
sheaf $\mathcal H$ on $X\times T$, flat over $T$, with
$[\mathcal H_{t_0}]\in\mathcal B$, there is an open neighborhood
$T_0\subset T$ of $t_0$ such that
\[
    [\mathcal H_t]\in\widehat{\mathcal B}
    \qquad\text{for every }t\in T_0.
\]
\end{corollary}

\

Now we start the proof of boundedness. We first reduce the full boundedness statement to the case of stable vector bundles:
\begin{lemma}\label{lem--reduce the full boundedness to bundle case}
    Suppose for every fixed rank, $c_1, c_2$, the set of isomorphism classes of $\mu$-stable holomorphic vector bundles with this topological data is bounded. Then Theorem \ref{thm-boundedness of semistable sheaves with fixed topology} is true.
\end{lemma}
\begin{proof}
    Let $\mathcal E$ be a $\mu$-semistable torsion-free sheaf with fixed topological data. We take one saturated Seshadri filtration $\{\mathcal E_i\}$, and denote by $\mathcal Q_i=\mathcal E_i/\mathcal E_{i-1}$ the corresponding stable factors. Set
    \[
    r_i\coloneqq\rk(\mathcal Q_i),\quad \ell_i\coloneqq\ell(\mathcal Q_i^{\vee\vee}/\mathcal Q_i),\quad a_i\coloneqq\frac{c_1(\mathcal Q_i)}{r_i},\quad a\coloneqq\frac{c_1(\mathcal E)}{r}.
    \]
    Here $\ell(\mathcal T)$ denotes the total length of a torsion sheaf $\mathcal T$ supported at isolated points. Then additivity of Chern classes gives
    \[
    \frac{\Delta(\mathcal E)}{r}=\sum_i\frac{\Delta(\mathcal Q_i^{\vee\vee})}{r_i}+2\sum_i\ell_i-\sum_ir_i(a_i-a)^2.
    \]
    The key observation is that every term on the right hand side is non-negative. The first term follows from Bogomolov's inequality applied to $\mathcal Q_i^{\vee\vee}$, and the second term is obvious.

    For the last term, we use the Hodge index theorem for compact K\"ahler surfaces. Since $\mu(\mathcal Q_i)=\mu(\mathcal E)$, we have
    \[
    (a_i-a)\cdot\omega=0.
    \]
    By Hodge index theorem, the intersection form on $H^{1,1}\cap\omega^{\perp}$ is negative-definite. Therefore the last term is also non-negative.

    Now since we fix the topology of $\mathcal E$, it follows that the classes $(a_i-a)$ lie in a bounded subset of $\omega^{\perp}$, thus only finitely possibilities for $c_1(\mathcal Q_i)$ can occur. Together with the bound on $\Delta(\mathcal Q_i^{\vee\vee})$, we see only finitely many possibilities for $c_2(\mathcal Q_i^{\vee\vee})$. Since $\mathcal Q_i^{\vee\vee}$ are stable vector bundles, our hypothesis implies that all such $\mathcal Q_i^{\vee\vee}$ form a bounded set. Using the uniform bound on  $\ell_i$, we can then conclude that all such $\mathcal Q_i$ belongs to a bounded set. By Proposition \ref{prop--boundedness for kernels, cokernels, images...}, all such $\mathcal E$ must also be bounded.
    
\end{proof}

\begin{proof}[Proof of Theorem \ref{thm-boundedness of semistable sheaves with fixed topology}]
    For any coherent sheaf $\mathcal F$ on $X$, we define
    \[
    e(\mathcal F)=-\ch_2(\mathcal F)\cdot[X]+\frac{(c_1(\mathcal F)\cdot[\omega])^2}{\rk(\mathcal F)\cdot[\omega]^2},\quad\nu(\mathcal F)=\rk(\mathcal F)+e(\mathcal F).
    \]
    It is clear that for any polystable vector bundle $\mathcal E$ with corresponding HYM connection $A$, 
    \[
    e(\mathcal E)=\frac{1}{8\pi^2}\Vert F_A\Vert_{L^2}\geq 0.
    \]
    A crucial property of this quantity is that it is discrete and monotonic along bubbling degeneration. More precisely, let $A_j$ be a sequence of HYM connections on $E$, and converges in the Uhlenbeck sense to an ideal HYM connection $(A_\infty,\mathcal C_\infty)$. Suppose $A_\infty$ defines a polystable holomorphic bundle $\mathcal E_\infty$ with one stable factor $\mathcal E_{\infty,i}$. Then 
    \[
    \nu(\mathcal E_{\infty,i})\leq\nu(\mathcal E_j)=\nu(E),
    \]
    with equality holds if and only if $A_\infty$ is irreducible and $\mathcal C_\infty=\varnothing$.
    
    Therefore we consider
    \[
   \mathcal S_N\coloneqq \{[\mathcal E]:\mathcal E\ \text{is a stable vector bundle and}\ \nu(\mathcal E)\leq N\}.
 \]
 We then induct on $N$ to see each $\mathcal S_N$ is bounded, which will finish the proof of Theorem \ref{thm-boundedness of semistable sheaves with fixed topology} by Lemma \ref{lem--reduce the full boundedness to bundle case}.

 We first observe that the set of determinant line bundles of $\mathcal S_N$ is bounded. Since $X$ is a compact K\"ahler manifold, $\mathrm{Pic}^0(X)$ is a compact torus. Thus it suffices to prove $c_1$ takes only finitely many values, which easily follows from Chern-Weil theory. Therefore boundedness of $S_N$ reduces to boundedness of stable bundles for each fixed rank, determinant and $c_2$. Note that on dimension four, this data also fixes the underlying complex vector bundle.

 Let $N_0$ be the minimum possible value of $N$. By discussion above, we can fix the determinant line bundle $\mathcal L$ and the underlying topological bundle $E$, and thus consider a fixed moduli space $M^*_{HYM}$ of smooth irreducible HYM connections on $E$ with determinant $\mathcal L$. It suffices to prove $M^*_{HYM}$ is bounded. By minimality of $N_0$ we see $M^*_{HYM}$ is sequentially compact, thus being compact by Theorem \ref{thm--uhlenbeck compactification is metrizable and compact}. The boundedness follows from the standard Kuranishi theory for holomorphic vector bundles. This proves the boundedness for $N_0$.

 Now assume $S_{N-1}$ is bounded. Again we can fix one moduli space $M^*_{HYM}$ and it suffices to prove it is bounded. We consider the set $\mathcal B$ of isomorphism classes of all algebraic limits of HYM connections in $M^*_{HYM}$ produced by Theorem \ref{thm-global convergence in a flat family}, for which either the Uhlenbeck convergence is not smooth or the limit connection is reducible. By monotonicity of $\nu$ and our induction assumption, it follows from Proposition \ref{prop--boundedness for kernels, cokernels, images...} that $\mathcal B$ is a bounded set. Let $\widehat{\mathcal B}\supset\mathcal B$ be the bounded set produced in Corollary \ref{cor--small perturbation of bounded set is again bounded}.

 We consider 
 \[
 N\coloneqq M^*_{HYM}\setminus\widehat{\mathcal B}.
 \]
 By Theorem \ref{thm-global convergence in a flat family} and defining property of $\widehat{\mathcal B}$, it is not hard to see the closure of $N$ is contained in $M^*_{HYM}$, thus being a compact domain in $M^*_{HYM}$ by Theorem \ref{thm--uhlenbeck compactification is metrizable and compact}. It follows that $N$ is a bounded set, so does $M^*_{HYM}$. The proof is now complete.
 
\end{proof}

\begin{remark}
    A similar induction argument also appears in \cite{toma01}.
\end{remark}
\begin{remark}
    We wonder if it is possible to avoid Donaldson's metrizablity theorem in the argument above. For instance, can we prove directly that the moduli space is second countable?
\end{remark}

\section{Towards the construction of compact analytic moduli spaces}\label{sec-discussions}

The moduli space of HYM connections has been studied extensively over the years and has deep connections to four-dimensional topology; we refer to \cite{DK97} for a detailed account. The geometry of moduli spaces can also reflect the complex structure of the underlying surfaces, leading to a beautiful classification result for class VII surfaces with small $b_2$ \cite{teleman05,teleman10}. 

In the complex geometric context, the most outstanding result is the Donaldson-Uhlenbeck-Yau theorem \cite{donaldson85,donaldson87,uhlenbeck-yau86}, also commonly known as the Hitchin-Kobayashi correspondence. It says that on any compact K\"ahler manifold, there exists a natural bijection between the moduli space of $\mu$-stable vector bundles and the moduli space of smooth irreducible HYM connections. Such a set theoretical bijection also turns out to be an isomorphism between \textit{real} analytic spaces \cite{fujiki-schumacher87,miyajima89}. In particular, $M^*_{HYM}$ is naturally endowed with a \textit{complex} analytic structure.

The basic philosophy underlying Hitchin-Kobayashi correspondence is that the metric geometry reflects the complex geometry, while the complex geometry restricts the metric geometry. Therefore, it is natural to expect that the comparison between sheaf moduli and HYM moduli should also be extended to the boundary of moduli spaces. When both the manifold and the polarization are algebraic, this was achieved in \cite{li1993} for surfaces and in \cite{GSTW2021} for arbitrary dimensions. In particular, in such algebraic framework, we know that the Donaldson-Uhlenbeck moduli space $\overline{M}_{HYM}$ is a compact complex analytic space. 

Given these discussions, there is a folklore conjecture about the analyticity of the Donaldson-Uhlenbeck boundary for general compact K\"ahler manifolds \textit{without} assuming algebraicity. The first modest question is to verify this for compact K\"ahler surfaces; for instance, see \cite{BTT2017,BTT2018}.
\begin{conjecture}\label{conj-complx structrue on the boundary}
     Let $(X,\omega)$ be a compact K\"ahler surface. The Donaldson-Uhlenbeck moduli space $\overline{M}_{HYM}$ is a complex analytic space, extending the natural complex analytic structure on $M^*_{HYM}$.
\end{conjecture}

The current paper suggests and supports a plausible folklore strategy for attacking this conjecture. The hope is that combined with the differential geometric techniques, we can still construct an algebraic moduli space for \textit{smoothable} semistable sheaves and then realize such a comparison. Based on the fundamental work in \cite{toma16,toma20,toma2021}, it remains to establish the following three steps.

\

\noindent\textbf{Step 1}. Near the Donaldson-Uhlenbeck boundary, choose a \textit{``polystable''} representative $\mathcal F_0$ mapped to the given ideal HYM connection. Then study the local analytic structure of the versal deformation space quotient by a suitable group of symmetries.

The existence of such a representative is guaranteed by Theorem \ref{thm-global convergence in a flat family}, while it may not be a genuine $\mu$-polystable sheaf. Based on Toma's semistable reduction theorem \cite{toma20} and Rees construction, one may also replace the algebraic limit in Theorem \ref{thm-global convergence in a flat family} by a GM-polystable sheaf still mapped to the given ideal HYM connection.

The harder step is to study the local analytic structure of the suitable quotient space. The existence of versal deformation is guaranteed by \cite{siu-trautmann81} under a fairly general hypothesis, so the key is to establish an analytic version of equivariant Luna's slicing on the versal deformation space of $\mathcal F_0$; cf. \cite[Proposition 2.15]{greb-toma2020}. While it is likely difficult to establish such a result in general, on surface case one may proceed with a length $1$ locally-free resolution constructed in \cite{schuster82}. We can try to deform the resolution together with its gluing data \cite[Proposition 3.1]{toma01}, then remove the freedom that merely changes the presentation of the same sheaf. Working with the deformed locally-free resolution may help us to handle the appearance of singularities.

\noindent\textbf{Step 2}. Glue the local analytic structures and construct a global proper moduli space for smoothable semistable sheaves.

This might be established by extending the fundamental work of \cite{AFS17} to the analytic setting. We also ask whether a more direct gluing approach is possible when the base is merely a surface.

\noindent\textbf{Step 3}. Study the structure of comparison map between sheaf moduli and gauge theoretical moduli, and thereby establishing Conjecture \ref{conj-complx structrue on the boundary}.

The fundamental difficulty is to understand equivalence relation on sheaf moduli, i.e. the construction of local holomorphic functions separating different sheaves. The proof in \cite{li1993,GSTW2021} relies on global structure of moduli space, thus does not extend immediately to general K\"ahler setting. Extensions of Cartan's criterion in \cite{mathieu00} might provide potential gain in general non-algebraic set-up.

\

We also remark that such a strategy goes back to the construction of $\mathbb Q$-Gorenstein smoothable $K$-semistable Fano moduli \cite{OSS16,odaka15,li-wang-xu19}, based on the analytic foundation established in \cite{DS1,CDS15,SSY18}. An analogous sheaf theoretical construction also appears in \cite{greb-toma2020,buchdahl-schumacher21,buchdahl-schumacher22}. In particular, this would also provide a new construction of proper moduli for semistable sheaves even in the algebraic setting.

\

On the other hand, as is well-known in the literature, in general we do not expect a construction of complex analytic moduli spaces in the non-K\"ahler category. For example, the Zariski openness of semistable sheaves fails for general Gauduchon metrics \cite{teleman08}, and the boundedness is also not well formulated there \cite[Remark 4.2]{pavel-toma24}. However, in certain cases complex structures can still be achieved; for instance see \cite{BTT2018}.

\appendix

\section{Proof of Theorem \ref{thm--separatedness for slope stability}}\label{sec--proof of separatedness for slope stability}

Recall that for any $\mu$-semistable torsion-free sheaf $\mathcal F$, the following pair is well-defined \cite[Appendix]{BTT2017}
\[
\Psi(\mathcal F)=(\gr^S(\mathcal F)^{\vee\vee},\mathcal C(\mathcal F)),
\]
where $\gr^S(\mathcal F)$ is the graded object of some Seshadri filtration and $\mathcal C(\mathcal F)$ is the effective codimension $2$ support cycle of the torsion sheaf $\gr^S(\mathcal F)^{\vee\vee}/\gr^S(\mathcal F)$. Denote by $R(\mathcal F)\coloneqq\gr^S(\mathcal F)^{\vee\vee}$. For a general torsion sheaf $\mathcal T$ with at least codimension $2$ support, we denote by $\mathcal C_{\mathcal T}$ its \textit{codimension 2} effective support cycle.

The following additivity should be well-known to experts:
\begin{proposition}\label{prop--additivity of semistable extension}
    Let $(X,\omega)$ be a compact K\"ahler manifold. If
        \[
        0\to \mathcal F\to\mathcal E\to\mathcal G\to 0
        \]
        is an exact sequence of torsion-free $\mu$-semistable sheaves of the same slope. Then
        \[
        R(\mathcal E)=R(\mathcal F)\oplus R(\mathcal G),\quad \mathcal C(\mathcal E)=\mathcal C(\mathcal F)+\mathcal C(\mathcal G).
        \]
\end{proposition}
\begin{proof}
    We prove by concatenating the filtration. We choose two saturated Seshadri filtrations of $\mathcal F$ and $\mathcal G$
    \[
    0=\mathcal F_0\subset\cdots\subset\mathcal F_r=\mathcal F,\quad 0=\mathcal G_0\subset\cdots\mathcal G_s=\mathcal G.
    \]
    Denote by $\pi$ the quotient map $\mathcal E\to\mathcal G$, and define $\pi^{-1}(\mathcal G_i)$ as the kernel of the natural surjective morphism
    \[
    \mathcal E\xrightarrow{\pi}\mathcal G\to\mathcal G/\mathcal G_i.
    \]
    It is easy to see that we have the follow two exact sequences
    \[
    0\to\mathcal F\to\pi^{-1}(\mathcal G_i)\to\mathcal G_i\to 0
    \]
    and
    \[
    0\to\pi^{-1}(\mathcal G_i)\to\mathcal E\to\mathcal G/\mathcal G_i\to 0.
    \]
    In particular, the second exact sequence shows that $\pi^{-1}(\mathcal G_i)$ is saturated in $\mathcal E$.

    Now we consider concatenated filtration
    \[
    0=\mathcal F_0\subset\cdots\subset\mathcal F_r=\mathcal F=\pi^{-1}(\mathcal G_0)\subset\pi^{-1}(\mathcal G_1)\subset\cdots\subset\pi^{-1}(\mathcal G_s)=\mathcal E.
    \]
    The successive quotients are 
    \[
    \mathcal F_i/\mathcal F_{i-1},\ 1\leq i\leq r,\quad \mathcal G_j/\mathcal G_{j-1},\ 1\leq j\leq s.
    \]
    In fact, the natural map 
    \[
    \pi^{-1}(\mathcal G_j)\to\mathcal G_j/\mathcal G_{j-1}
    \]
    is surjective with kernel $\pi^{-1}(\mathcal G_{j-1})$, so
    \[
    \pi^{-1}(\mathcal G_j)/\pi^{-1}(\mathcal G_{j-1})\cong\mathcal G_j/\mathcal G_{j-1}.
    \]
 Therefore, this concatenated filtration is indeed a Seshadri filtration of $\mathcal E$. By taking double dual and support cycle, we conclude the identification.

\end{proof}

We will also use the following simple fact:
\begin{lemma}\label{lem--extension of semistable same slope sheaves is semistable}
    Let 
    \[
    0\to\mathcal F\to\mathcal E\to\mathcal G\to 0
    \]
    be an exact sequence of coherent sheaves, such that both $\mathcal F$ and $\mathcal G$ are torsion-free and $\mu$-semistable. Moreover, assume they have the same slope. Then $\mathcal E$ is torsion-free and $\mu$-semistable of the same slope.
\end{lemma}
\begin{proof}
    We first prove torsion-freeness. Any torsion element in $\mathcal E$ would be mapped to zero in the torsion-free sheaf $\mathcal G$, so lies in $\mathcal F$. Since $\mathcal F$ is also torsion-free, it must be zero.

    Denote by $\mu$ the slope of $\mathcal F$ and $\mathcal G$. By additivity of degree and rank, we get
    \[
    \deg\mathcal E=\deg\mathcal F+\deg\mathcal G=\mu\cdot(\rk\mathcal F+\rk\mathcal G)=\mu\cdot\rk\mathcal E,
    \]
    so $\mu(\mathcal E)=\mu$.

    Finally we prove the semistability of $\mathcal E$. Let $\mathcal E_0\subset\mathcal E$ be a subsheaf, and we consider
    \[
    \mathcal F_0=\mathcal E_0\cap\mathcal F,\quad \mathcal G_0=\mathrm{im}(\mathcal E_0\to\mathcal G).
    \]
    There is a natural exact sequence
    \[
    0\to\mathcal F_0\to\mathcal E_0\to\mathcal G_0\to 0,
    \]
    so
    \[
    \det\mathcal E_0=\deg\mathcal F_0+\deg\mathcal G_0.
    \]
    On the other hand, by semistability of $\mathcal F,\mathcal G$, we have
    \[
    \deg\mathcal F_0\leq\mu\cdot \rk\mathcal F_0,\quad\deg\mathcal G_0\leq\mu\cdot\rk\mathcal G_0.
    \]
    It follows that
    \[
    \deg\mathcal E_0\leq\mu\cdot(\rk\mathcal F_0+\rk\mathcal G_0)=\mu\cdot\rk\mathcal E_0,
    \]
    i.e. $\mu(\mathcal E_0)\leq\mu=\mu(\mathcal E)$, so $\mathcal E$ is semistable.
\end{proof}

Now we prove Theorem \ref{thm--separatedness for slope stability}.
\begin{proof}[Proof of Theorem \ref{thm--separatedness for slope stability}]
    As in the first step of the proof of \cite[Theorem 6.4]{toma20}, by possibly shrinking the disk, we can assume there exists a morphism $\varphi:\mathcal E\to\mathcal F$, such that $\varphi_t$ is an isomorphism for every $t\in\Delta^*$. It follows that $\varphi$ is injective, since $\ker\varphi$ is a subsheaf of $\mathcal E$ supported in $X\times\{0\}$, which must be zero by flatness of $\mathcal E$. Hence $\mathcal E$ is identified with a subsheaf of $\mathcal F$.

    The quotient $\mathcal F/\mathcal E$ is also supported in $X\times\{0\}$, so by compactness of $X$ we can find $N>0$, such that $t^N\mathcal F\subset\mathcal E$. We now define a finite sequence of subsheaves
    \[
    \mathcal E^{(j)}\coloneqq\mathcal E\cap t^j\mathcal F,\ 0\leq j\leq N.
    \]
    Each sheaf is a subsheaf of $\mathcal E$, so cannot have $t$-torsion. It follows that they are all flat over $\Delta$.

    For any $j\geq 0$, the multiplication map 
    \[
    \mathcal F\xrightarrow{\cdot t^j}t^j\mathcal F
    \]
    is an isomorphism of sheaves over $X\times\Delta$. The injectivity follows from flatness of $\mathcal F$, while the surjectivity follows from the definition of $t^j\mathcal F$. In particular, the central fiber of $t^j\mathcal F$ is isomorphic to 
    \[
    \frac{t^j\mathcal F}{t^{j+1}\mathcal F}\cong\frac{\mathcal F}{t\mathcal F}=\mathcal F_0,
    \]
    which is non-zero. Note that this is different from the natural inclusion map $t^j\mathcal F\to\mathcal F$, where the induced map on central fiber is zero.

    Consider $E_j=\mathcal E^{(j)}\vert_{X\times\{0\}}$. Using the above identification $t^j\mathcal F\cong\mathcal F$, the inclusion $\mathcal E^{(j)}\subset t^j\mathcal F$ induces a morphism between central fibers
    \[
    \varphi_j: E_j\to\mathcal F_0.
    \]
    We now induct on $j$ to prove that $E_j$ is torsion-free and semistable of the common slope $\mu=\mu(\mathcal E_0)=\mu(\mathcal F_0)$. Moreover, $\Psi(E_j)=\Psi(\mathcal E_0)$. Since $\mathcal E^{(N)}=t^N\mathcal F\cong\mathcal F$, the result then follows.

    The assertion is trivial when $j=0$. Now assume the assertion for $E_j$, and we study $E_{j+1}$. Denote by $K_j=\ker\varphi_j$ and $I_j=\Im\varphi_j$. By induction assumption, both are subsheaves of torsion-free sheaf, so they are all torsion-free. Let's assume both have positive rank, since the other cases are easier to handle. By semistability of $E_j$ and $\mathcal F_0$, we have two-side inequalities
    \[
    \mu(I_j)\geq\mu(E_j)=\mu,\quad \mu(I_j)\leq\mu(\mathcal F_0)=\mu,
    \]
    so $\mu(I_j)=\mu$. It follows that $\mu(K_j)=\mu$. Since $\mathcal F_0$ is semistable of slope $\mu$ and $I_j$ is its subsheaf of the same slope, it must also be semistable. Likewise the inclusion $K_j\subset E_j$ implies that $K_j$ is semistable.

    Now we relate $K_j, I_j$ to $E_{j+1}$. Denote by $\iota: X\times\{0\}\hookrightarrow X\times\Delta$ the natural inclusion. Consider the natural morphism
   \[
   \mathcal E^{(j)}
\hookrightarrow t^j\mathcal F
\twoheadrightarrow
\frac{t^j\mathcal F}{t^{j+1}\mathcal F}
\cong \iota_*\mathcal F_0,
\]
whose image is $\iota_*I_j$. The kernel is thus
\[
\mathcal E^{(j)}\cap t^{j+1}\mathcal F=\mathcal E^{(j+1)}.
\]
It follows that we have the exact sequence
\[
0\to\mathcal E^{(j+1)}\to\mathcal E^{(j)}\to\iota_* I_j\to 0.
\]
By restricting this sequence to the central fiber, we get
\[
0\to\mathrm{Tor}_1^{\mathcal O_{\Delta,0}}(\iota_*I_j,\mathcal O_{\Delta,0}/(t))\to E_{j+1}\to E_j\to I_j\to 0.
\]
By general theory, we know that
\[
\mathrm{Tor}_1^{\mathcal O_{\Delta,0}}(\iota_*I_j,\mathcal O_{\Delta,0}/(t))\cong\ker(t:\iota_* I_j\to\iota_* I_j)\cong I_j,
\]
since $I_j$ is a subsheaf of $\mathcal F_0=\mathcal F/t\mathcal F$. Using that $K_j=\ker(E_j\to I_j)$, we thus obtain an exact sequence
\begin{equation}\label{equa--exact seq 1 in separatedness}
    0\to I_j\to E_{j+1}\to K_j\to 0.
\end{equation}
It follows from Lemma \ref{lem--extension of semistable same slope sheaves is semistable} that $E_{j+1}$ is torsion-free and semistable of slope $\mu$. On the other hand, we have the defining exact sequence
\begin{equation}\label{equa--exact seq 2 in separatedness}
    0\to K_j\to E_j\to I_j\to 0.
\end{equation}
Applying Proposition \ref{prop--additivity of semistable extension} to both \eqref{equa--exact seq 1 in separatedness} and \eqref{equa--exact seq 2 in separatedness}, we conclude that $\Psi(E_{j+1})=\Psi(E_j)$.

\end{proof}

\section{Proof of Theorem \ref{thm--local deformation to global deformation}}\label{sec--proof of local-to-global passage}

In this section, we give a detailed account of Douady's theory on deformation of compact analytic spaces over possibly infinite dimensional parameter base, and then prove the local-to-global passage theorem.

\subsection{Idea of proof}\label{subsec--idea of proof for local-to-global passage theorem}
We start by explaining the underlying idea of the proof.

Recall the discussion in Section \ref{subsec--global construction of flat families}, that under our assumption on local projective dimension, each $R_\alpha$ is assumed to be injective as a sheaf morphism. Therefore, the liftings $B_{\alpha\beta}$ and $H_{\alpha\beta\gamma}$ are uniquely determined by $\{q_\alpha\}$ and $A_{\alpha\beta}$. It is then natural to expect that under such a hypothesis, those liftings also converge to the corresponding limit lifting.
\begin{lemma}\label{lem--convergence after injective relation}
    Let $R_j$ and $T_j$ be holomorphic matrices on an open subset $V$ of a complex manifold. Suppose 
    \[
    R_j\to R_0,\quad R_jT_j\to R_0 T_0\ \text{both in }\ C^{\infty}_{loc}(V).
    \]
   If $R_0:\mathcal O_V^k\to\mathcal O_V^N$ is injective as a sheaf morphism, then $T_j\to T_0$ in $C^{\infty}_{loc}(V)$.
\end{lemma}
\begin{proof}
    We write $R_j$ as an $N\times k$ matrix and $T_j$ as a $k\times m$ matrix.
The case $k=0$ is immediate. It suffices to work with a neighborhood of a fix $x\in V$. Since $R_0$ is an injective
map of free sheaves on a complex manifold, some $k\times k$ row minor
$M_0$ has determinant $d_0$ which is nonzero as a \textit{holomorphic germ} at $x$.
Let $M_j$ be the same row minor of $R_j$.

If $d_0(x)\ne0$, the inverses $M_j^{-1}$ exist and are uniformly
bounded on a sufficiently small neighborhood of $x$, for all large $j$, so the result is immediate. 

Suppose instead that $d_0(x)=0$. Choose local coordinates
$(z',w)$ centered at $x$ such that $d_0(0,w)$ is not identically zero.
There are a small disc $\Delta_r=\{|w|<r\}$ and a polydisc $D'$ in
the $z'$-coordinates such that
\[
 \overline{D'\times\Delta_r}\subset V,
 \qquad |d_0|\ge c>0
 \quad\hbox{on }K:=\overline{D'}\times\partial\Delta_r.
\]
Indeed, first choose a circle on which $d_0(0,w)$ has no zeros, and
then shrink $D'$. Uniform convergence of $R_j$ on $K$ implies that on such $K$, $M_j^{-1}$ exists and is uniformly bounded for all large $j$.

By our hypothesis, it is easy to see that 
\[
R_j(T_j-T_0)\to 0\quad\hbox{locally uniformly on }V.
\]
Selecting the rows defining $M_j$ and multiplying by $M_j^{-1}$ gives
$(T_j-T_0)\to0$ uniformly on $K$. Finally we apply the maximum principle on $z^{\prime}$-slice to each
entry as a function of $w$, for every $z'\in D'$. This gives uniform
convergence on $D'\times\Delta_r$.
\end{proof}

\begin{corollary}\label{cor--convergence of liftings B and H}
    Under the hypothesis of Theorem \ref{thm--local deformation to global deformation}, we have
    \[
   \begin{aligned}
      & B_{\alpha\beta,j}\to B_{\alpha\beta, 0}&& \text{in}\ C^{\infty}_{loc}(V_{\alpha\beta}),\\
       & H_{\alpha\beta\gamma,j}\to H_{\alpha\beta\gamma,0}&& \text{in}\ C^{\infty}_{loc}(V_{\alpha\beta\gamma}).
   \end{aligned}
    \]
\end{corollary}
\begin{proof}
   Apply Lemma~\ref{lem--convergence after injective relation} to
\[
\begin{aligned}
 R_{\alpha,j}B_{\alpha\beta,j}
 &=A_{\alpha\beta,j}R_{\beta,j}
   \longrightarrow R_{\alpha,0}B_{\alpha\beta,0},\\
 R_{\alpha,j}H_{\alpha\beta\gamma,j}
 &=A_{\alpha\beta,j}A_{\beta\gamma,j}-A_{\alpha\gamma,j}
   \longrightarrow R_{\alpha,0}H_{\alpha\beta\gamma,0},
\end{aligned}
\]
in $C^\infty_{\mathrm{loc}}$ on the corresponding overlaps.
\end{proof}

Given Corollary \ref{cor--convergence of liftings B and H}, the proof of Theorem \ref{thm--local deformation to global deformation} will go as follows. By working with a refinement $\{U_\alpha\}$ of $\{V_\alpha\}$, we can consider the Banach space of holomorphic matrices on $U_\alpha, U_{\alpha\beta}$ and $U_{\alpha\beta\gamma}$ endowed with the $L^{\infty}$-norm. We then take the finite product to get a Banach space $\mathcal E$, so each collection of \textit{arbitrary} holomorphic matrices $(R, A, B,H)$ defines an element in $\mathcal E$. On the other hand, as we have seen in Section \ref{sec--proof of local-to-global passage}, the genuine sheaf construction requires the following compatibility conditions:
\begin{equation}\label{equa--defining equation for mathfrak B}
    \begin{aligned}
   & A_{\alpha\beta}R_\beta= R_\alpha B_{\alpha\beta},&&\text{on}\ U_{\alpha\beta},\\
    & A_{\alpha\beta}A_{\beta\gamma}-A_{\alpha\gamma}=R_\alpha H_{\alpha\beta\gamma},&&\text{on}\ U_{\alpha\beta\gamma},\\
    & A_{\alpha\alpha}=\mathrm{Id},&&\text{on}\ U_\alpha.
\end{aligned}
\end{equation}
We regard this as a holomorphic defining equation on $\mathcal E$, and obtain a Banach analytic space $\mathfrak B$. Our convergent data $(R_j, A_j, B_j, H_j)$ naturally define a sequence of points $b_j\in\mathfrak B$ converging to $b_0$. Moreover, the defining equation \eqref{equa--defining equation for mathfrak B} tautologically defines a sheaf $\mathcal G_{\mathfrak B}$ on $X\times\mathfrak B$, such that
\[
\mathcal G_{\mathfrak B}\vert_{X\times\{b_0\}}=\mathcal G_0,\quad \mathcal G_{\mathfrak B}\vert_{X\times\{b_j\}}=\mathcal G_j.
\]
One would be able to check directly that, $\mathcal G_{\mathfrak B}$ is \textit{anaflat} over $\mathfrak B$, a notion introduced by Douady \cite{douady66} generalizing coherence and flatness in finite dimensional setting.

The fundamental work in \cite{douady74} establishes the existence of versal deformations in complex analytic geometry. More importantly, the versality holds even for anaflat families over Banach analytic parameter spaces. This allows us to produce a finite dimensional flat family from the tautological sheaf $\mathcal G_{\mathfrak B}$ over $X\times\mathfrak B$.

There is one more technical point we need to address. The versal deformation in \cite{douady74} was constructed for finite dimensional compact analytic spaces, while we are concerned with sheaf theoretical deformations. To pass from coherent sheaf to analytic space, we use a classical construction called \textit{square-zero extension}, also known as \textit{Nagata's principle of idealization} \cite[Chapter I, §1]{Nagata62}. The basic idea is as follows. Let $R$ be a commutative unit ring and $M$ be a $R$-module. We define a ring structure on $R[M]\coloneqq R\oplus M$ via
\[
(a,u)\cdot(b,v)\coloneqq(ab,av+bu).
\]
The inclusion of $R$ into $R[M]$ and the projection to the first factor are natural ring homomorphisms
\[
p^*:R\to R[M],\quad\iota^*: R[M]\to R,\quad \iota^*\circ p^*=\mathrm{id}.
\]
Moreover, the module $M$ now sits inside $R[M]$ as an ideal with $M^2=0$:
\[
\{0\}\oplus M\hookrightarrow R\oplus M,
\]
which can also be covered as
\[
M\cong\ker(R[M]\to R).
\]
This transfers between $R$-modules and ideals in $R[M]$.

 Globalizing this, given a complex analytic space $X$ and a coherent analytic sheaf $\mathcal F$ over $X$, we can define a ringed space
 \[
 X[\mathcal F]\coloneqq(X,\mathcal O_X[\mathcal F]),
 \]
 which is isomorphic to the first infinitesimal neighborhood of the zero section $\iota: X\hookrightarrow\mathrm{V}(\mathcal F)$, thus being an analytic space. There are natural morphisms
 \[
 X\xrightarrow{\iota}X[\mathcal F]\xrightarrow{p}X,\quad p\circ\iota=\mathrm{id},
 \]
 so the underlying sheaf can be recovered as
 \[
 \mathcal F
=
\ker\!\left(
p_*\iota^*:p_*\mathcal O_{X[\mathcal F]}\longrightarrow\mathcal O_X
\right).
 \]

\subsection{Review of Douady's theory}

In this subsection, we record the fundamental ingredients from \cite{douady66,douady74} and \cite{pourcin69}.

We start with the definition of Banach analytic spaces, following \cite[\S3]{douady66}. The key distinction to finite dimensional setting comes from the fact that it is not enough to encode just the information of \textit{scalar} holomorphic functions. Instead, Douady defined the structure sheaf by considering $G$-valued holomorphic functions for \textit{all} Banach space $G$. He called them $\mathcal K$-functored spaces, using the appropriate categorical language.

To be precise, the local model of Douady's Banach analytic space is as follows. Let $E,F$ be two Banach spaces over $\mathbb C$, $U$ an open subset of $E$, and $f: U\to F$ is an analytic map. The underlying topological space is considered to be $X=f^{-1}(0)$. 

For any Banach space $G$, the $G$-valued holomophic functions on $X$ is defined to be
\[
\Phi(U,G)\coloneqq\frac{\mathcal H(U,G)}{\{\Lambda f:\Lambda\in\mathcal H(U, L(F,G))\}},
\]
where $\mathcal H(U,G)$ denotes the set of holomorphic maps between $U,G$ and $\mathcal L(F,G)$ denotes the Banach space of all bounded linear operators between $F,G$. The composition is defined to be
\[
\Lambda f: U\to G,\quad x\mapsto\Lambda(x)(f(x))\in G.
\]
That is, any two representatives differ by a function vanishing on $X$. 

To encode all Banach-space-valued functions, we introduce:
\begin{definition}
    Let $\mathcal K$ be a category. A $\mathcal K$-functored space is a topological space $X$ equipped with a functor $\Phi$ from $\mathcal K$ to the category of sheaves of sets on $X$.
\end{definition}

\begin{definition}
    Let $\mathcal K$ be the category of open subsets of Banach spaces and analytic mappings. Then a Banach analytic space is a $\mathcal K$-functored space, such that every point has an open neighborhood isomorphic to a model described above.
\end{definition}

Note that any Banach analytic space is naturally a ringed space by taking the structure sheaf $\Phi(\mathbb C)$. However, the ringed space structure is not enough to capture the analytic structure as defined above. As a simple example, we take
\[
D=\bigl(\{0\},\mathbf C[\varepsilon]/(\varepsilon^2)\bigr),
\]
and $G$ is a \textit{non-reflexive} Banach space. Take $b\in G^{**}\setminus G$ and define
\[
\alpha_b:\mathcal O_{G,0}\longrightarrow
\mathbf C[\varepsilon]/(\varepsilon^2),
\qquad
\alpha_b(h)=h(0)+b\bigl(dh(0)\bigr)\varepsilon.
\]
One can verify that $\alpha_b$ is a morphism between ringed spaces, but is not induced from an analytic map. We refer to \cite{EW17} for detailed discussions.

We next define the notion of anaflatness over Banach analytic parameter base. For this purpose, we need to first define anaflatness for open subsets of $\mathbb C^n$.

\begin{definition}[\S8.3, Définition in \cite{douady66}]
    Let $U$ be an open subset of $\mathbb C^n$, $S$ be a Banach analytic space and $\mathcal F$ be a sheaf of $\mathcal O_{U\times S}$-modules. We say $\mathcal F$ is $S$-anaflat, if for every $(x,s)\in U\times S$, there exists a finite locally-free resolution 
    \[
    0\to\mathcal L_p\to\cdots\to\mathcal L_0\to\mathcal F\to 0
    \]
    of $\mathcal F$ on a neighborhood of $(x,s)\in U\times S$, such that the fiberwise complex
    \[
    0\to\mathcal L_p(s)\to\cdots\to\mathcal L_0(s)\to\mathcal F(s)\to 0
    \]
    is exact.
\end{definition}
Here we follow the notation of \cite{douady66}, i.e. $\mathcal L(s)$ denotes the restriction of $\mathcal L$ to $U\times\{s\}$, which helps to distinguish it from a stalk.

The following suggests its connection to the usual notion of flatness:
\begin{proposition}[\S8.3, Proposition 1 in \cite{douady66}]
    Suppose $S$ is finite dimensional, then $\mathcal F$ is $S$-anaflat if and only if it is coherent and flat in the usual sense.
\end{proposition}

Now for general finite dimensional analytic space, we introduce:
\begin{definition}[\S8.7, Définition in \cite{douady66}]
    Let $X$ be a finite dimensional analytic space and $S$ be a Banach analytic space. We say a sheaf $\mathcal F$ of $\mathcal O_{X\times S}$-modules is $S$-anaflat, if for every $x\in X$, there exists an open neighborhood $X_1$ of $x$ embedded as a closed subspace of an open subset $U$ of $\mathbb C^n$, such that the sheaf $(\iota_S)_*\mathcal F$ is $S$-anaflat over $U\times S$.
\end{definition}

This definition is well-defined due to the following lemma:
\begin{lemma}[\S8.7, Lemma 2 in \cite{douady66}]\label{lem--B2 well-definedness of anaflatness}
    Let $U^{\prime}$ ($U^{\prime\prime}$, respectively) be an open subset of $\mathbb C^{n^{\prime}}$ ($\mathbb C^{n^{\prime\prime}}$, respectively). Take $U=U^{\prime}\times U^{\prime\prime}\subset\mathbb C^{n^{\prime}+n^{\prime\prime}}$. Let $\mathcal F^{\prime}$ be a sheaf of $\mathcal O_{U^{\prime}\times S}$-modules on $U^{\prime}\times S$ and $\mathcal F=(\iota_S)_*\mathcal F^{\prime}$. Then $\mathcal F$ is $S$-anaflat if and only if $\mathcal F^{\prime}$ is $S$-anaflat.
\end{lemma}

We will use the following basic property of anaflat sheaves. The reader should compare this with the usual property of flatness.
\begin{proposition}[\S8.3, Proposition 2 and Proposition 3 in \cite{douady66}]\label{prop--property of anflat sheaf}
    The following holds:
    \begin{itemize}
        \item [(1)] Let $\mathcal F,\mathcal F^{\prime},\mathcal F^{\prime\prime}$ be $S$-anaflat sheaves on $U\times S$, and
    \[
    \mathcal F\to\mathcal F^{\prime}\to\mathcal F^{\prime\prime}
    \]
    be a complex. Fix any $(x,s)\in U\times S$. If 
    \[
    \mathcal F(s)\to\mathcal F^{\prime}(s)\to\mathcal F^{\prime\prime}(s)
    \]
    is exact at $x$, then the original sequence of exact in a neighborhood of $(x,s)$.
    \item [(2)] Let $0\to\mathcal F^{\prime}\to\mathcal F\to\mathcal F^{\prime\prime}\to 0$ be an exact sequence of sheaves of $\mathcal O_{U\times S}$-modules on $U\times S$. If both $\mathcal F^{\prime},\mathcal F^{\prime\prime}$ are $S$-anaflat, then $\mathcal F$ is anaflat. On the other hand, if both $\mathcal F^{\prime},\mathcal F$ are $S$-anaflat, then $\mathcal F^{\prime\prime}$ is $S$-anflat if and only if $\mathcal F^{\prime}(s)\to\mathcal F(s)$ is injective for every $s\in S$.
    \end{itemize}

\end{proposition}

Therefore, a deformation of a finite dimensional analytic space would be a Banach analytic space, whose structure sheaf is anaflat over the parameter space. The first main result we will use is the following existence theorem for versal deformations:
\begin{theorem}[VII. 8, Théorème principal in \cite{douady74}]\label{thm--existence of versal deformation for compact analytic spaces}
    Let $X$ be a compact finite dimensional complex analytic space. Then its versal deformation exists as a finite dimensional complex analytic space germ, i.e. there exists a finite dimensional analytic space germ $(R,0)$ and a flat family of compact analytic spaces $\mathfrak X\to (R,0)$ with an isomorphism $\mathfrak X_0\cong X$, such that the following holds. Given any family $\mathfrak Y\to (S,0)$ proper and anflat over the Banach analytic parameter base $S$ with $\mathfrak Y_0\cong X$, there exists a morphism $h:(S,0)\to(R,0)$, such that $h^*\mathfrak X\cong\mathfrak Y$ in a neighborhood of $0\in S$.
\end{theorem}

 Note that it is important for us that the versality holds even for Banach analytic parameter bases.

 The second main result we need is the following representability theorem established in \cite{douady66}. One should regard this as the analytic version of the Quot schemes.
 \begin{theorem}[\S9.7, Théorème 1 and \S9.8, Théorème 2 in \cite{douady66}]
     Let $X$ be a finite-dimensional complex space and $\mathcal E$ be a coherent sheaf on $X$. Define a functor
\[
\mathfrak Q_{X,\mathcal E}:\mathsf{BanAn}^{\mathrm{op}}\longrightarrow\mathsf{Sets}
\]by
\[
\mathfrak Q_{X,\mathcal E}(T)
=
\left\{
p_X^*\mathcal E\twoheadrightarrow\mathcal Q:
\begin{array}{l}
\mathcal Q\text{ is }T\text{-anaflat},\\
\operatorname{Supp}\mathcal Q\to T\text{ is proper}
\end{array}
\right\}\big/\cong,
\]
where quotients are considered up to isomorphism compatible with the quotient map. Then $\mathfrak Q_{X,\mathcal E}$ is represented by an analytic space $H_X(\mathcal E)$, i.e. for any Banach analytic space $T$, there is a natural bijection:
\[
\mathrm{Hom}_{\mathsf{BanAn}}\!\left(T,H_X(\mathcal E)\right)
\;\cong\;
\mathfrak Q_{X,\mathcal E}(T).
\]
Moreover, the representing space $H_X(\mathcal E)$ is Hausdorff and finite-dimensional at every point.
 \end{theorem}

Here we do not require that the ambient space $X$ is compact, instead $\operatorname{Supp}\mathcal Q$ is required to be proper over $T$. Similar to Grothendieck's construction in algebraic geometry, there are two consequences of the existence of this analytic Quot space. First, by taking $\mathcal E=\mathcal O_X$, we obtain the analytic version of Hilbert schemes, now known as the Douady space:
 \begin{theorem}\label{thm--existence of absolute Douady space}
Let $X$ be a finite-dimensional Hausdorff complex space, and let
$\mathsf{BanAn}$ denote the category of complex Banach analytic spaces.

Define the functor
\[
\mathfrak D_X :
\mathsf{BanAn}^{\mathrm{op}} \longrightarrow \mathsf{Sets}
\]
by
\[
\mathfrak D_X(T)
=
\left\{
\Gamma \subset X \times T :
\begin{array}{l}
\Gamma \text{ is a closed analytic subspace},\\
\mathcal O_\Gamma \text{ is } T\text{-anaflat},\\
\Gamma \to T \text{ is proper}
\end{array}
\right\}.
\]

Then $\mathfrak D_X$ is represented by a Hausdorff complex space
$D(X)$, which is finite-dimensional at every point. Therefore, for every Banach
analytic space $T$, there is a natural bijection
\[
\operatorname{Hom}_{\mathsf{BanAn}}(T,D(X))
\;\cong\;
\mathfrak D_X(T),
\]
induced by pulling back the universal closed analytic subspace
\[
\mathcal U \subset X\times D(X).
\]
\end{theorem}

The second consequence is the construction of the morphism space. We can embed the morphism space $\mathsf{Mor}(X,Y)$ into the Hilbert scheme $\mathsf{Hilb}_{X\times Y}$ as an open subset via the graph construction: a subspace $\Gamma\subset X\times Y$ is a graph precisely when the natural projection $\Gamma\to X$ is an isomorphism. However, what we do need is the following relative version, which is the consequence of the relative Douady's theorem established in \cite[Théorème 1]{pourcin69}.

\begin{theorem}\label{thm--existence of relative morphism space}
Let $\mathsf{An}_S$ denote the category of finite dimensional Hausdorff complex spaces. Let $X \to S$ and $Y \to S$ be objects
in $\mathsf{An}_S$, and assume that $X \to S$ is proper and flat.

For each complex space $T \to S$, set
\[
X_T := X \times_S T,
\qquad
Y_T := Y \times_S T.
\]
Define the functor
\[
\mathfrak M_{X,Y/S} :
(\mathsf{An}_S)^{\mathrm{op}} \longrightarrow \mathsf{Sets},
\qquad
T \longmapsto \operatorname{Hom}_T(X_T,Y_T).
\]

Then $\mathfrak M_{X,Y/S}$ is represented by a complex space
\[
M = \operatorname{Mor}_S(X,Y) \longrightarrow S,
\]
which is Hausdorff and finite-dimensional at every point. Therefore, for every complex space $T \to S$, there is a natural bijection
\[
\operatorname{Hom}_S(T,M)
\;\cong\;
\operatorname{Hom}_T(X_T,Y_T).
\]
These bijections are induced by pulling back the universal
$M$-morphism
\[
\mathrm{ev} : X_M \longrightarrow Y_M.
\]
\end{theorem}

 Theorem \ref{thm--existence of versal deformation for compact analytic spaces}, Theorem \ref{thm--existence of absolute Douady space} and Theorem \ref{thm--existence of relative morphism space} will be all we need from the literature.

\subsection{Proof of main theorem}

Now we are ready to provide the details of the proof. The overall proof is wordy and basically fills in details for the outline in Section \ref{subsec--idea of proof for local-to-global passage theorem}.

\begin{proof}[Proof of Theorem~\ref{thm--local deformation to global deformation}]
All analytic spaces below are allowed to be nonreduced. As before, for any sheaf $\mathcal F$ on $Z\times S$, we write $\mathcal F(s)$ for the restriction to $Z\times\{s\}$ when
it is useful to distinguish it from a stalk.

\medskip
\noindent\textit{Step 1. Construction of the Banach analytic parameter space.}

Choose a finite cover of $X$ by connected coordinate polydiscs
$U_a$ and denote by the refinement map $a\mapsto\rho(a)$ such that
$\overline{U_a}\subset V_{\rho(a)}$. We then restrict the given
presentations and overlap matrices to this cover, and relabel its
indices by $\alpha,\beta,\gamma$. 

Now we fix this refinement. It is clear that the restricted
matrices are bounded on their respective domains, and by Corollary \ref{cor--convergence of liftings B and H}, all four
collections $R_j,A_j,B_j,H_j$ converge in the corresponding $L^{\infty}$-norms.

For an open set $W\subset X$, put
\[
 \mathcal H^{p,q}(W)
 :=H^\infty\bigl(W,\operatorname{Hom}(\mathbb C^q,\mathbb C^p)\bigr),
 \qquad \|F\|_W:=\sup_{x\in W}\|F(x)\|.
\]
By general theory, this defines a Banach space. The desired Banach space in
Section~\ref{subsec--idea of proof for local-to-global passage theorem}
is the following finite product:
\begin{equation}\label{equa--B3 Banach product}
\begin{aligned}
 \mathcal E={}&
 \prod_\alpha\mathcal H^{N_\alpha,k_\alpha}(U_\alpha)
 \times\prod_{\alpha,\beta}
       \mathcal H^{N_\alpha,N_\beta}(U_{\alpha\beta})\\
 &\times\prod_{\alpha,\beta}
       \mathcal H^{k_\alpha,k_\beta}(U_{\alpha\beta})
 \times\prod_{\alpha,\beta,\gamma}
       \mathcal H^{k_\alpha,N_\gamma}(U_{\alpha\beta\gamma}).
\end{aligned}
\end{equation}
An element of $\mathcal E$ will be denoted by $(R,A,B,H)$. 

Let $\mathcal T$ be the
finite product of the spaces
\[
 \mathcal H^{N_\alpha,k_\beta}(U_{\alpha\beta}),\quad
 \mathcal H^{N_\alpha,N_\gamma}(U_{\alpha\beta\gamma}),\quad
 \mathcal H^{N_\alpha,N_\alpha}(U_\alpha),
\]
and define $\Phi:\mathcal E\to\mathcal T$ by
\begin{equation}\label{equa--B3 Banach equations}
 \Phi(R,A,B,H)=
 \left(
 \begin{array}{c}
 (A_{\alpha\beta}R_\beta-R_\alpha B_{\alpha\beta})_{\alpha,\beta}\\
 (A_{\alpha\beta}A_{\beta\gamma}-A_{\alpha\gamma}
                 -R_\alpha H_{\alpha\beta\gamma})_{\alpha,\beta,\gamma}\\
 (A_{\alpha\alpha}-\mathrm{Id})_\alpha
 \end{array}
 \right).
\end{equation}
It is clear that $\Phi$ is a
holomorphic polynomial map of Banach spaces, so its zero locus defines a Banach analytic space.

Let $b_0=(R_0,A_0,B_0,H_0)$, and $P\subset\mathcal E$ be a
sufficiently small open neighborhood of $b_0$. We are concerned with the Banach analytic space germ
\[
 \mathfrak B:=\Phi^{-1}(0)\cap P.
\]
By construction, the given data define points
\begin{equation}\label{equa--B3 Banach convergence}
 b_j=(R_j,A_j,B_j,H_j)\in\mathfrak B,
 \qquad b_j\longrightarrow b_0,
\end{equation}
for all sufficiently large $j$.

\medskip
\noindent\textit{Step 2. The tautological anaflat family of sheaves.}

This is standard. We define
\[
 \mathcal G_\alpha
 :=\operatorname{coker}\left(
 R_\alpha:\mathcal O_{U_\alpha\times\mathfrak B}^{k_\alpha}
       \longrightarrow
       \mathcal O_{U_\alpha\times\mathfrak B}^{N_\alpha}
 \right).
\]
The first equation in \eqref{equa--B3 Banach equations} makes
$A_{\alpha\beta}$ descend to a morphism
$\overline A_{\alpha\beta}:\mathcal G_\beta\to\mathcal G_\alpha$
on the overlap, while the other equations give
\[
 \overline A_{\alpha\beta}\overline A_{\beta\gamma}
     =\overline A_{\alpha\gamma},
 \qquad \overline A_{\alpha\alpha}=\mathrm{Id}.
\]
Moreover, by taking $\gamma=\alpha$ we see $\overline A_{\alpha\beta}$ and $\overline A_{\beta\alpha}$
are inverse. It follows that these maps glue the local cokernels to a
sheaf $\mathcal G_{\mathfrak B}$ on $X\times\mathfrak B$. From the construction, it's clear that
\begin{equation}\label{equa--B3 tautological fibres}
 \mathcal G_{\mathfrak B}(b_0)\cong\mathcal G_0,
 \qquad \mathcal G_{\mathfrak B}(b_j)\cong\mathcal G_j.
\end{equation}

We then verify the anaflatness. For each $\alpha$ with $k_\alpha>0$,
the injective sheaf map $R_{\alpha,0}$ has a $k_\alpha\times
k_\alpha$ minor which is nonzero at some point $x_\alpha\in
U_\alpha$. After shrinking $P$, that minor remains nonzero at
$x_\alpha$ for every parameter $b\in\mathfrak B$. Since
$U_\alpha$ is connected, its determinant is not identically zero
on any nonempty open subset of $U_\alpha$. Thus
\[
 R_\alpha(b):\mathcal O_{U_\alpha}^{k_\alpha}
       \longrightarrow\mathcal O_{U_\alpha}^{N_\alpha}
\]
is injective as a sheaf map for every $b\in\mathfrak B$. For
$k_\alpha=0$ there is nothing to prove.

We now apply Proposition \ref{prop--property of anflat sheaf} to the complex
\[
 0\longrightarrow\mathcal O^{k_\alpha}
        \xrightarrow{R_\alpha}\mathcal O^{N_\alpha}.
\]
It follows from fiberwise injectivity that $R_\alpha$ is an injective morphism between $\mathcal O_{U_\alpha\times\mathfrak B}$-sheaves. We thus have a finite resolution
sequence
\begin{equation}\label{equa--B3 universal resolution}
 0\longrightarrow\mathcal O_{U_\alpha\times\mathfrak B}^{k_\alpha}
 \xrightarrow{R_\alpha}
 \mathcal O_{U_\alpha\times\mathfrak B}^{N_\alpha}
 \longrightarrow\mathcal G_\alpha\longrightarrow0
\end{equation}
whose restriction to every fiber is exact. Hence $\mathcal G_{\mathfrak B}$ is $\mathfrak B$-anaflat.

\medskip
\noindent\textit{Step 3. The square-zero extension as a proper anaflat family.}

Define the sheaf of algebras
\[
 \mathcal A_{\mathfrak B}
 :=\mathcal O_{X\times\mathfrak B}\oplus\mathcal G_{\mathfrak B},
 \qquad (f,u)(g,v)=(fg,fv+gu).
\]
We spell out why it defines an analytic space with underlying topological space $X\times\vert\mathfrak B\vert$.

In
$U_\alpha\times\mathbb C^{N_\alpha}\times\mathfrak B$, we impose the equations
\begin{equation}\label{equa--B3 square-zero equations}
 \begin{aligned}
 w_lw_m&=0 &&(1\le l,m\le N_\alpha),\\
 \sum_{l=1}^{N_\alpha}(R_\alpha)_{lk}(x,b)w_l&=0
     &&(1\le k\le k_\alpha).
 \end{aligned}
\end{equation}
Modulo the first quadratic equations, a holomorphic function has only its
constant and linear terms in $w$: the constant terms idenfity the $\mathcal O_{U_\alpha\times\mathfrak B}$ part, while the second linear equation identify the
linear part with $\mathcal G_\alpha$. Thus the resulting local
algebra is $\mathcal O_{U_\alpha\times\mathfrak B}\oplus
\mathcal G_\alpha$. The isomorphisms $\overline A_{\alpha\beta}$
give algebra isomorphisms on overlaps, satisfying the cocycle
condition. It follows that we can glue these local spaces to obtain:
\[
 \mathfrak Y:=(X\times\mathfrak B)[\mathcal G_{\mathfrak B}].
\]
The inclusion of the first summand and the augmentation give maps
\begin{equation}\label{equa--B3 Banach retraction}
 X\times\mathfrak B\xrightarrow{\iota_{\mathfrak B}}
 \mathfrak Y\xrightarrow{p_{\mathfrak B}}X\times\mathfrak B,
 \qquad p_{\mathfrak B}\circ\iota_{\mathfrak B}=\mathrm{id}.
\end{equation}
Since $X$ is compact, $\mathfrak Y\to\mathfrak B$ is a proper morphism.

We then check the anaflatness. Let
$j_\alpha:U_\alpha\times\mathfrak B\hookrightarrow
U_\alpha\times\mathbb C^{N_\alpha}\times\mathfrak B$
be the zero section. Regarding structure sheaves as modules on this
ambient space, there is an exact sequence
\[
 0\longrightarrow (j_\alpha)_*\mathcal G_\alpha
 \longrightarrow\mathcal O_{\mathfrak Y|_{U_\alpha}}
 \longrightarrow(j_\alpha)_*\mathcal O_{U_\alpha\times\mathfrak B}
 \longrightarrow0.
\]
We claim that both $(j_\alpha)_*\mathcal G_\alpha$ and $(j_\alpha)_*\mathcal O_{U_\alpha\times\mathfrak B}$ are anaflat over $\mathfrak B$. Given this, It follows from Proposition \ref{prop--property of anflat sheaf} that $\mathfrak Y\to\mathfrak B$ is a proper anaflat family
of compact analytic spaces, whose fiber at $b_j$ is given by $\mathfrak Y_{b_j}=X[\mathcal G_j]$. 

The anaflatness of $(j_\alpha)_*\mathcal O_{U_\alpha\times\mathfrak B}$ follows from Lemma \ref{lem--B2 well-definedness of anaflatness}, since $j_\alpha$ is a closed embedding. For $(j_\alpha)_*\mathcal G_\alpha$, we use the exact sequence
\[
0\longrightarrow(({j_\alpha})_*\mathcal O_{U_\alpha\times\mathfrak B})^{k_\alpha}
\xrightarrow{(j_\alpha)_*R_\alpha}
((j_\alpha)_*\mathcal O_{U_\alpha\times\mathfrak B})^{N_\alpha}
\longrightarrow (j_\alpha)_*\mathcal G_\alpha
\longrightarrow0.
\]
The first and second term are anaflat by Lemma \ref{lem--B2 well-definedness of anaflatness} again. Since the fiberwise map $R_\alpha$ is injective, by Proposition \ref{prop--property of anflat sheaf} we know the coker $(j_\alpha)_*\mathcal G_\alpha$ is also anaflat.

\medskip
\noindent\textit{Step 4. Reduction to finite-dimensional family.}

We now apply Theorem~\ref{thm--existence of versal deformation for compact analytic spaces}
to the compact analytic space $X[\mathcal G_0]$ and the anaflat family
$\mathfrak Y\to\mathfrak B$. It follows that we can find a finite-dimensional germ
$(R,0)$, a proper flat family $\pi:\mathfrak X\to R$, a
holomorphic map of germs
\[
 h:(\mathfrak B,b_0)\longrightarrow(R,0),
\]
and an isomorphism $\theta:h^*\mathfrak X\xrightarrow{\sim}
\mathfrak Y$ over $\mathfrak B$. Transporting
\eqref{equa--B3 Banach retraction} by $\theta$, one get holomorphic maps
\[
\widetilde p:=p_{\mathfrak B}\circ\theta:
h^*\mathfrak X\longrightarrow X\times\mathfrak B,\quad \widetilde\iota:=\theta^{-1}\circ\iota_{\mathfrak B}:
X\times\mathfrak B\longrightarrow h^*\mathfrak X,\quad \tilde p\circ\tilde\iota=\mathrm{id}.
\]
In particular, by restricting to fibers, we get holomorphic maps
\begin{equation}\label{equa--B3 transported maps}
 p_{b_j}:\mathfrak X_{h(b_j)}\longrightarrow X,
 \qquad \iota_{b_j}:X\longrightarrow\mathfrak X_{h(b_j)},
 \qquad p_{b_j}\circ\iota_{b_j}=\mathrm{id}_X.
\end{equation}

Recall that the sheaf can be recovered from the projection and zero section map. For our purpose, we need a family version of this inverse construction, so our next goal is to incorporate $(p_{b_j},\iota_{b_j})$ together
into a finite-dimensional parameter space. 

Both $\mathfrak X\to R$ and $X\times R\to R$ are proper and
flat, so by Theorem~\ref{thm--existence of relative morphism space}, the morphism spaces
\[
 \begin{aligned}
 M_p&:=\operatorname{Mor}_R(\mathfrak X,X\times R),\\
 M_\iota&:=\operatorname{Mor}_R(X\times R,\mathfrak X)
 \end{aligned}
\]
are finite dimensional analytic spaces. We will denote by $(r,p)$ a point in $M_p$. In fact, applying the universal property to the one-point space $T=\mathrm{pt}_r$ mapped to $R$ with image $r$, we then get
\[
\operatorname{Hom}_R(\mathrm{pt}_r,M_p)
\;\cong\;
\operatorname{Hom}(\mathfrak X_r,X).
\]
The left side is precisely the set of points of $M_p$ over $r$, while the right side is the set of holomorphic maps
\[
p:\mathfrak X_r\to X.
\]
Given this notation, the universal morphism
\[
\mathcal P:\mathfrak X\times_R M_p\to X\times M_p
\]
over $(r,p)\in M_p$ is precisely $p:\mathfrak X_r\to X$. Similarly, we denote by $(r,\iota)$ a point in $M_\iota$.

Put 
\[
M:=M_p\times_R M_\iota.
\]
We pull back both universal morphisms to $M$, so obtain
\[
X\times M\xrightarrow{\iota_M}\mathfrak X_M\xrightarrow{p_M}X\times M.
\]
Therefore, $p_M\circ\iota_M$ gives an $M$-morphism of
$X\times M$ to itself. The representing property of
$\operatorname{Mor}(X,X)$ thus gives a holomorphic map
\[
 c:M\longrightarrow\operatorname{Mor}(X,X),
 \qquad (r,p,\iota)\longmapsto p\circ\iota.
\]
Consider the fiber of $c$ over $\mathrm{id}_X$ as an analytic subspace, and let $(S,0)$ be the germ at $(0,p_{b_0},\iota_{b_0})$ of this fiber. Set $\mathfrak X_S:=\mathfrak X\times_R S$, which is
proper and flat over $S$. It then carries universal maps
\begin{equation}\label{equa--B3 finite retraction}
 X\times S\xrightarrow{\iota}\mathfrak X_S
 \xrightarrow{p}X\times S,
 \qquad p\circ\iota=\mathrm{id}_{X\times S},
\end{equation}
whose restriction to $h(b_j)$ is $(p_{b_j},\iota_{b_j})$.

Finally, we need to show that 
$s_j\coloneqq (h(b_j),p_{b_j},\iota_{b_j})$ converges to $0$ in $S$. This seems to be a naive thing to show, as long as one can get a continuous classifying map $f:\mathfrak B\to S$ mapping $b$ to $(h(b),p_b,\iota_b)$. As we did not find such a statement for Banach analytic bases (cf. the statement of Theorem \ref{thm--existence of relative morphism space} from \cite{pourcin69}, where we restrict to the category of finite dimensional analytic spaces), we spell out some details by using Theorem \ref{thm--existence of absolute Douady space} which allows Banach analytic bases. In the end, this should be the same as establishing the Banach version of universal property for (relative) morphism spaces.

Consider the fixed finite-dimensional complex space
\[
 Q:=\mathfrak X\times X,
\]
with its map to $R$ through $\pi:\mathfrak X\to R$. There are two closed analytic
subspaces of $Q\times\mathfrak B$: the graph of the transported
projection $\tilde p$ and the graph of the transported zero section $\tilde\iota$. The fibers over $b$ are
\[
 \begin{aligned}
 \Gamma_{p,b}&=\{(z,p_b(z)):z\in\mathfrak X_{h(b)}\},\\
 \Gamma_{\iota,b}&=\{(\iota_b(x),x):x\in X\}.
 \end{aligned}
\]
Globally the first graph is isomorphic over
$\mathfrak B$ to $h^*\mathfrak X\cong\mathfrak Y$, and the
second to $X\times\mathfrak B$. Thus both graph families are
proper and anaflat over $\mathfrak B$.

Theorem~\ref{thm--existence of absolute Douady space} now gives
holomorphic maps
\[
 \gamma_p,\gamma_\iota:\mathfrak B\longrightarrow D(Q)
\]
classifying these graph families. It follows from
\eqref{equa--B3 Banach convergence} that
\begin{equation}\label{equa--B3 graph convergence}
 \bigl(h(b_j),\gamma_p(b_j),\gamma_\iota(b_j)\bigr)
 \longrightarrow
 \bigl(0,\gamma_p(b_0),\gamma_\iota(b_0)\bigr)
 \quad\hbox{in }R\times D(Q)\times D(Q).
\end{equation}

We recall how this convergence detects the topology of $M$.
The relative Douady space $D(Q/R)$ is locally a closed analytic
subspace of $R\times D(Q)$. It imposes the condition that a
compact subspace $\Gamma\subset Q$ lie in the fiber $Q_r$ as an
\emph{analytic subspace}. This is the construction in
\cite[\S4]{pourcin69}.

Inside $D(Q/R)$, the locus where the projection
$\Gamma\to\mathfrak X_r$ is an isomorphism is open and is
identified with $M_p$. Similarly, the locus where
$\Gamma\to X$ is an isomorphism is open and is identified with
$M_\iota$. Thus the graph
construction give locally closed embeddings
\[
 M_p\hookrightarrow R\times D(Q),
 \qquad M_\iota\hookrightarrow R\times D(Q).
\]
Consequently $M$ has the local topology induced
from $R\times D(Q)\times D(Q)$. The same is true for its analytic
subspace $S$. Each triple in the left side of
\eqref{equa--B3 graph convergence} corresponds to
$(h(b_j),p_{b_j},\iota_{b_j})$, so we have the convergence
\begin{equation}\label{equa--B3 finite parameter convergence}
 s_j:=(h(b_j),p_{b_j},\iota_{b_j})\in S,
 \qquad s_j\longrightarrow0.
\end{equation}

\medskip
\noindent\textit{Step 5. Recover the sheaf family.}

We first shrink $S$ so that the projection $p$ in
\eqref{equa--B3 finite retraction} is finite. To see this, note that the map $p$ is proper: its graph is a closed embedding
into $\mathfrak X_S\times_S(X\times S)$, and projection from
this product to $X\times S$ is a base change of the proper map
$\mathfrak X_S\to S$. 

The central map is isomorphic to
\[
 X[\mathcal G_0]\longrightarrow X,
\]
which is finite since $\mathcal G_0$ is coherent. By the upper semicontinuity of fiber dimensions \cite[Chapter II, \S8.1, Proposition 8.2]{Demailly},
the locus $Z\subset\mathfrak X_S$ where $p$ has a positive
dimensional fiber germ is closed and misses the central fiber.
Its image in $S$ is closed because $\mathfrak X_S\to S$ is
proper. Shrinking $S$ to avoid this image, we can arrange that all fibers of $p$ are zero-dimensional. Thus $p$ is a proper holomorphic map with zero-dimensional
fibers, hence being finite.

By Grauert's coherence theorem, the direct image
$\mathcal A:=p_*\mathcal O_{\mathfrak X_S}$ is a coherent
$\mathcal O_{X\times S}$-algebra. Define
\begin{equation}\label{equa--B3 final sheaf}
 \mathcal G:=\ker\left(
 p_*\iota^*:\mathcal A\longrightarrow\mathcal O_{X\times S}
 \right).
\end{equation}
The identity $p\circ\iota=\mathrm{id}$ means that the unit map
$p^*:\mathcal O_{X\times S}\to\mathcal A$ splits the
augmentation $p_*\iota^*$. Therefore, as modules we have
\begin{equation}\label{equa--B3 final splitting}
 \mathcal A\cong\mathcal O_{X\times S}\oplus\mathcal G.
\end{equation}
In particular $\mathcal G$ is coherent. Moreover, by finiteness of $p$, we know that $\mathcal A$ is flat over $S$. As a direct summand, it follows that $\mathcal G$ is also flat over $S$.

Finally, since finite direct image commutes with base change, together with the splitting
\eqref{equa--B3 final splitting}, we can identify the restriction
of \eqref{equa--B3 final sheaf} to any fiber with
\[
 \mathcal G|_{X\times\{s\}}
 \cong\ker\left(
 (p_s)_*\mathcal O_{(\mathfrak X_S)_s}
       \xrightarrow{(p_s)_*\iota_s^*}\mathcal O_X
 \right).
\]
It follows that
\[
 \mathcal G|_{X\times\{0\}}\cong\mathcal G_0,
 \qquad
 \mathcal G|_{X\times\{s_j\}}\cong\mathcal G_j
 \quad(j\gg1).
\]
\end{proof}

\bibliographystyle{alpha}
\bibliography{ref.bib}

\begin{thebibliography}{GSTW21}

\bibitem[AFS17]{AFS17}
Jarod Alper, Maksym Fedorchuk, and David~Ishii Smyth.
\newblock {Second flip in the Hassett--Keel program: existence of good moduli spaces}.
\newblock {\em Compositio Mathematica}, 153(8):1584--1609, 2017.

\bibitem[AH61]{AH61}
M.F. Atiyah and F.~Hirzebruch.
\newblock {Vector bundles and homogeneous spaces}.
\newblock In {\em Differential geormetry. Proceedings of the Symposium in Pure Mathematics}, volume~3, pages 7--38, 1961.

\bibitem[Ati67]{Atiyah67}
M.~F. Atiyah.
\newblock {\em {K}-theory}.
\newblock W. A. Benjamin, Inc., New York--Amsterdam, 1967.
\newblock Lecture notes by D. W. Anderson.

\bibitem[Ban90a]{bando1990a}
Shigetoshi Bando.
\newblock {Bubbling out of Einstein manifolds}.
\newblock {\em Tohoku Mathematical Journal, Second Series}, 42(2):205--216, 1990.

\bibitem[Ban90b]{bando1990b}
Shigetoshi Bando.
\newblock {Correction and addition: bubbling out of Einstein manifolds}.
\newblock {\em Tohoku Mathematical Journal, Second Series}, 42(4):587--588, 1990.

\bibitem[Bin80]{bingener80}
J{\"u}rgen Bingener.
\newblock {Darstellbarkeitskriterien f{\"u}r analytische Funktoren}.
\newblock In {\em Annales scientifiques de l'{\'E}cole Normale Sup{\'e}rieure}, volume~13, pages 317--347, 1980.

\bibitem[BS21]{buchdahl-schumacher21}
Nicholas Buchdahl and Georg Schumacher.
\newblock {An analytic application of geometric invariant theory}.
\newblock {\em Journal of Geometry and Physics}, 165:104237, 2021.

\bibitem[BS22]{buchdahl-schumacher22}
Nicholas Buchdahl and Georg Schumacher.
\newblock {An analytic application of Geometric Invariant Theory II: Coarse moduli spaces}.
\newblock {\em Journal of Geometry and Physics}, 175:104467, 2022.

\bibitem[BTT17]{BTT2017}
Nicholas Buchdahl, Andrei Teleman, and Matei Toma.
\newblock A continuity theorem for families of sheaves on complex surfaces.
\newblock {\em Journal of topology}, 10(4):995--1028, 2017.

\bibitem[BTT18]{BTT2018}
Nicholas Buchdahl, Andrei Teleman, and Matei Toma.
\newblock {On the Donaldson--Uhlenbeck compactification of instanton moduli spaces on class VII surfaces}.
\newblock {\em The Quarterly Journal of Mathematics}, 69(4):1423--1473, 2018.

\bibitem[CDS15]{CDS15}
Xiuxiong Chen, Simon Donaldson, and Song Sun.
\newblock {K{\"a}hler-Einstein metrics on Fano manifolds. I, II, III.}
\newblock {\em Journal of the American Mathematical Society}, 28(1):183--278, 2015.

\bibitem[CS21]{chen-sun2021}
Xuemiao Chen and Song Sun.
\newblock {Reflexive sheaves, Hermitian--Yang--Mills connections, and tangent cones}.
\newblock {\em Inventiones mathematicae}, 225(1):73--129, 2021.

\bibitem[Dem]{Demailly}
Jean-Pierre Demailly.
\newblock {\em {Complex analytic and differential geometry}}.

\bibitem[DK97]{DK97}
Simon~Kirwan Donaldson and Peter~B Kronheimer.
\newblock {\em The geometry of four-manifolds}.
\newblock Oxford university press, 1997.

\bibitem[Don83]{donaldson83}
Simon~K Donaldson.
\newblock An application of gauge theory to four-dimensional topology.
\newblock {\em Journal of Differential Geometry}, 18(2):279--315, 1983.

\bibitem[Don84]{donaldson84}
Simon~K Donaldson.
\newblock Instantons and geometric invariant theory.
\newblock {\em Communications in Mathematical Physics}, 93(4):453--460, 1984.

\bibitem[Don85]{donaldson85}
Simon~K Donaldson.
\newblock {Anti self-dual Yang-Mills connections over complex algebraic surfaces and stable vector bundles}.
\newblock {\em Proceedings of the London Mathematical Society}, 3(1):1--26, 1985.

\bibitem[Don86]{donaldson86}
Simon~K Donaldson.
\newblock Connections, cohomology and the intersection forms of 4-manifolds.
\newblock {\em Journal of Differential Geometry}, 24(3):275--341, 1986.

\bibitem[Don87]{donaldson87}
Simon~K Donaldson.
\newblock Infinite determinants, stable bundles and curvature.
\newblock {\em Duke Mathematical Journal}, 54(1):231--247, 1987.

\bibitem[Don88]{donaldson}
Simon~K Donaldson.
\newblock {Compactification and completion of Yang-Mills moduli spaces}.
\newblock In {\em Differential Geometry: Proceedings of the 3rd International Symposium, held at Pe{\~n}iscola, Spain, June 5--12, 1988}, pages 145--160. Springer, 1988.

\bibitem[Dou66]{douady66}
Adrien Douady.
\newblock Le probl{\`e}me des modules pour les sous-espaces analytiques compacts d'un espace analytique donn{\'e}.
\newblock In {\em Annales de l'institut Fourier}, volume~16, pages 1--95, 1966.

\bibitem[Dou74]{douady74}
Adrien Douady.
\newblock Le probl{\`e}me des modules locaux pour les espaces {$\mathbb C$}-analytiques compacts.
\newblock In {\em Annales scientifiques de l'{\'E}cole Normale Sup{\'e}rieure}, volume~7, pages 569--602, 1974.

\bibitem[DS14]{DS1}
Simon Donaldson and Song Sun.
\newblock {Gromov--Hausdorff limits of K{\"a}hler manifolds and algebraic geometry}.
\newblock {\em Acta Mathematica}, 213(1):63--106, 2014.

\bibitem[DS17]{DS2}
Simon Donaldson and Song Sun.
\newblock {Gromov--Hausdorff limits of K{\"a}hler manifolds and algebraic geometry, II}.
\newblock {\em Journal of Differential Geometry}, 107(2):327--371, 2017.

\bibitem[DW26]{DW26}
Anuk Dayaprema and Alex Waldron.
\newblock {Parabolic Gap Theorems for the Yang-Mills Energy}.
\newblock {\em Mathematical Physics, Analysis and Geometry}, 29(1):12, 2026.

\bibitem[EW17]{EW17}
Michel Egeileh and Tilman Wurzbacher.
\newblock Infinite-dimensional manifolds as ringed spaces.
\newblock {\em Publications of the Research Institute for Mathematical Sciences}, 53(1):187--209, 2017.

\bibitem[Fle81a]{flenner81b}
Hubert Flenner.
\newblock {Ein kriterium f{\"u}r die offenheit der versalit{\"a}t}.
\newblock {\em Mathematische Zeitschrift}, 178(4):449--473, 1981.

\bibitem[Fle81b]{flenner81a}
Hubert Flenner.
\newblock {Eine Bemerkung {\"u}ber relative Ext-Garben}.
\newblock {\em Mathematische Annalen}, 258(2):175--182, 1981.

\bibitem[FS87]{fujiki-schumacher87}
Akira Fujiki and Georg Schumacher.
\newblock {The moduli space of Hermite-Einstein bundles on a compact K{\"a}hler manifold}.
\newblock {\em Proceedings of the Japan Academy, Series A, Mathematical Sciences}, 63(3):69--72, 1987.

\bibitem[FU12]{freed-uhlenbeck}
Daniel~S Freed and Karen~K Uhlenbeck.
\newblock {\em Instantons and four-manifolds}, volume~1.
\newblock Springer Science \& Business Media, 2012.

\bibitem[GP97]{groisser-parker97}
David Groisser and Thomas~H Parker.
\newblock {Sharp decay estimates for Yang--Mills fields}.
\newblock {\em Communications in Analysis and Geometry}, 5(3):439--474, 1997.

\bibitem[GR79]{grauert-remmert79}
Hans Grauert and Reinhold Remmert.
\newblock {\em {Theory of Stein spaces}}.
\newblock Springer Science \& Business Media, 1979.

\bibitem[Gra58]{grauert58}
Hans Grauert.
\newblock {On Levi's problem and the imbedding of real-analytic manifolds}.
\newblock {\em Annals of Mathematics}, 68(2):460--472, 1958.

\bibitem[Gro60]{grothendieck60}
Alexander Grothendieck.
\newblock {Techniques de construction et th{\'e}oremes d’existence en g{\'e}om{\'e}trie alg{\'e}brique. IV. Les sch{\'e}mas de Hilbert}.
\newblock {\em S{\'e}minaire Bourbaki}, 6(221):249--276, 1960.

\bibitem[Gro97]{groisser1997}
David Groisser.
\newblock {The $L^2$ metric in gauge theory: an introduction and some applications}.
\newblock {\em Banach Center Publications}, 39(1):317--329, 1997.

\bibitem[GSTW21]{GSTW2021}
Daniel Greb, Benjamin Sibley, Matei Toma, and Richard Wentworth.
\newblock {Complex algebraic compactifications of the moduli space of Hermitian Yang--Mills connections on a projective manifold}.
\newblock {\em Geometry \& Topology}, 25(4):1719--1818, 2021.

\bibitem[GT20]{greb-toma2020}
Daniel Greb and Matei Toma.
\newblock {Moduli spaces of sheaves that are semistable with respect to a K{\"a}hler polarisation}.
\newblock {\em Journal de l’{\'E}cole polytechnique—Math{\'e}matiques}, 7:233--261, 2020.

\bibitem[Li93]{li1993}
Jun Li.
\newblock {Algebraic geometric interpretation of Donaldson's polynomial invariants}.
\newblock {\em Journal of differential geometry}, 37(2):417--466, 1993.

\bibitem[Li20]{li20}
Yang Li.
\newblock {Local Nahm Transform and Singularity Formation of ASD Connections}.
\newblock {\em Communications in Mathematical Physics}, 375(2):1041--1078, 2020.

\bibitem[LWX19]{li-wang-xu19}
Chi Li, Xiaowei Wang, and Chenyang Xu.
\newblock {On the proper moduli spaces of smoothable K{\"a}hler--Einstein Fano varieties}.
\newblock {\em Duke Math. J.}, 168(8):1387--1459, 2019.

\bibitem[Mat00]{mathieu00}
David Mathieu.
\newblock {Universal reparametrization of a family of cycles: a new approach to meromorphic equivalence relations}.
\newblock In {\em Annales de l'institut Fourier}, volume~50, pages 1155--1189, 2000.

\bibitem[Miy89]{miyajima89}
Kimio Miyajima.
\newblock {Kuranishi family of vector bundles and algebraic description of the moduli space of Einstein-Hermitian connections}.
\newblock {\em Publications of the Research Institute for Mathematical Sciences}, 25(2):301--320, 1989.

\bibitem[Nag62]{Nagata62}
Masayoshi Nagata.
\newblock {\em {Local rings}}.
\newblock Interscience Publishers, 1962.

\bibitem[Nak99]{nakajima99}
Hiraku Nakajima.
\newblock {\em {Lectures on Hilbert schemes of points on surfaces}}, volume~18.
\newblock American Mathematical Soc., 1999.

\bibitem[NV19]{NV19}
Aaron Naber and Daniele Valtorta.
\newblock {Energy identity for stationary Yang Mills}.
\newblock {\em Inventiones mathematicae}, 216(3):847--925, 2019.

\bibitem[Oda15]{odaka15}
Yuji Odaka.
\newblock {Compact moduli spaces of K{\"a}hler--Einstein Fano varieties}.
\newblock {\em Publications of the Research Institute for Mathematical Sciences}, 51(3):549--565, 2015.

\bibitem[OSS16]{OSS16}
Yuji Odaka, Cristiano Spotti, and Song Sun.
\newblock {Compact moduli spaces of del Pezzo surfaces and K{\"a}hler--Einstein metrics}.
\newblock {\em Journal of Differential Geometry}, 102(1):127--172, 2016.

\bibitem[Pou69]{pourcin69}
Genevieve Pourcin.
\newblock Th{\'e}or{\`e}me de {D}ouady au-dessus de $s$.
\newblock {\em Annali della Scuola Normale Superiore di Pisa-Scienze Fisiche e Matematiche}, 23(3):451--459, 1969.

\bibitem[PT24]{pavel-toma24}
Mihai Pavel and Matei Toma.
\newblock Slope-semistability and moduli of coherent sheaves: a survey.
\newblock {\em arXiv preprint arXiv:2407.13485}, 2024.

\bibitem[Rad93]{rade93}
Johan Rade.
\newblock {Decay estimates for Yang-Mills fields: two new proofs}.
\newblock {\em Global analysis in modern mathematics (Orono, ME, 1991; Waltham, MA, 1992)}, pages 91--105, 1993.

\bibitem[Sch82]{schuster82}
Hans~Werner Schuster.
\newblock {Locally free resolutions of coherent sheaves on surfaces.}
\newblock {\em J. Reine Angew. Math.}, (337):159--165, 1982.

\bibitem[SSY18]{SSY18}
Cristiano Spotti, Song Sun, and Chengjian Yao.
\newblock {Existence and deformations of K{\"a}hler--Einstein metrics on smoothable $\mathbb Q$-Fano varieties}.
\newblock {\em Duke Math. J.}, 165(16):3043--3083, 2018.

\bibitem[ST81]{siu-trautmann81}
Yum-Tong Siu and G{\"u}nther Trautmann.
\newblock {\em Deformations of coherent analytic sheaves with compact supports}, volume 238.
\newblock American Mathematical Soc., 1981.

\bibitem[Sun25]{sun25}
Song Sun.
\newblock {Bubbling of Kähler-Einstein metrics}.
\newblock {\em Pure and Applied Mathematics Quarterly}, 21:1317--1348, 2025.

\bibitem[SW15]{sibley-wentworth15}
Benjamin Sibley and Richard~A Wentworth.
\newblock {Analytic cycles, Bott--Chern forms, and singular sets for the Yang--Mills flow on K{\"a}hler manifolds}.
\newblock {\em Advances in mathematics}, 279:501--531, 2015.

\bibitem[SZ23]{sun-zhang2023}
Song Sun and Junsheng Zhang.
\newblock {No semistability at infinity for Calabi-Yau metrics asymptotic to cones}.
\newblock {\em Inventiones mathematicae}, 233(1):461--494, 2023.

\bibitem[Tel05]{teleman05}
Andrei Teleman.
\newblock {Donaldson theory on non-K{\"a}hlerian surfaces and class VII surfaces with $b_2=1$}.
\newblock {\em Inventiones mathematicae}, 162(3):493--521, 2005.

\bibitem[Tel08]{teleman08}
Andrei Teleman.
\newblock Families of holomorphic bundles.
\newblock {\em Communications in Contemporary Mathematics}, 10(04):523--551, 2008.

\bibitem[Tel10]{teleman10}
Andrei Teleman.
\newblock {Instantons and curves on class VII surfaces}.
\newblock {\em Annals of Mathematics}, pages 1749--1804, 2010.

\bibitem[Tom01]{toma01}
Matei Toma.
\newblock {Compact moduli spaces of stable sheaves over non-algebraic surfaces}.
\newblock {\em Documenta Mathematica}, 6:9--27, 2001.

\bibitem[Tom16]{toma16}
Matei Toma.
\newblock {Bounded sets of sheaves on K{\"a}hler manifolds}.
\newblock {\em Journal f{\"u}r die reine und angewandte Mathematik (Crelles Journal)}, 2016(710):77--93, 2016.

\bibitem[Tom20]{toma20}
Matei Toma.
\newblock Properness criteria for families of coherent analytic sheaves.
\newblock {\em Algebraic Geometry}, 7(4):486--502, 2020.

\bibitem[Tom21]{toma2021}
Matei Toma.
\newblock Bounded sets of sheaves on relative analytic spaces.
\newblock {\em Annales Henri Lebesgue}, 4:1531--1563, 2021.

\bibitem[TT]{TM}
A.~Teleman and M.~Toma.
\newblock Semistable sheaves on compact complex surfaces.
\newblock {\em work in progress}.

\bibitem[UY86]{uhlenbeck-yau86}
Karen Uhlenbeck and Shing-Tung Yau.
\newblock {On the existence of Hermitian-Yang-Mills connections in stable vector bundles}.
\newblock {\em Communications on Pure and Applied Mathematics}, 39(S1):S257--S293, 1986.

\bibitem[Wal16]{waldron16}
Alex Waldron.
\newblock {Instantons and singularities in the Yang--Mills flow}.
\newblock {\em Calculus of Variations and Partial Differential Equations}, 55(5):113, 2016.

\bibitem[Zha25]{zhang2025}
Junsheng Zhang.
\newblock {On polynomial convergence to tangent cones for singular K{\"a}hler--Einstein metrics}.
\newblock {\em Journal f{\"u}r die reine und angewandte Mathematik (Crelles Journal)}, 2025(826):271--302, 2025.

\bibitem[ZZ26]{ZZ26}
Junsheng Zhang and Keshu Zhou.
\newblock {On the geometry of non-collapsed polarized cscK surfaces}.
\newblock {\em arXiv preprint arXiv:2606.02816}, 2026.

\end{thebibliography}

\end{document}